\documentclass[12pt]{amsart}
\usepackage{amsmath}

\normalsize
\usepackage{amssymb,amsmath,amsthm,epsfig}

\newcommand{\lbl}[1]{\label{#1}}

\newtheorem{theo}{Theorem}[section]
\newtheorem{lem}[theo]{Lemma}
\newtheorem{remark}[theo]{Remark}
\newtheorem{definition[theo]}{Definition}

\newcommand\ee{\end{equation}}
\newcommand\bes{\begin{eqnarray}}
\newcommand\ees{\end{eqnarray}}
\newcommand\bess{\begin{eqnarray*}}
\newcommand\eess{\end{eqnarray*}}

\newcommand\bR{{\mathbb R}}

\begin{document}
\setlength{\baselineskip}{16pt} \pagestyle{myheadings}

\title[Spreading in a shifting environment ]
{Spreading in a shifting environment modeled by the diffusive logistic equation with a free boundary}
\thanks{}
\author[ Y. Du, L. Wei and L. Zhou]{Yihong Du$^\dag$,   Lei Wei$^\ddag$ and Ling Zhou$^\S$\;}
\thanks{$^\dag$School of Science and Technology,
University of New England, Armidale, NSW 2351, Australia.\\ Email:
ydu@\allowbreak turing.\allowbreak une.\allowbreak edu.\allowbreak
au.}
\thanks{$^\ddag$School of Mathematics and Statistics, Jiangsu Normal University, Xuzhou 221116,  China.\\
 Email: wlxznu@163.com.}
\thanks{$^\S$School of Mathematical Science, Yangzhou University, Yangzhou 225002, China.\\ Email: zhoul@yzu.edu.cn.}
\thanks{Y. Du was supported by the Australian Research Council,  L. Wei was supported by NSFC (11271167), and L. Zhou was supported by NSFC (11401515). }


\date{\today}

\begin{abstract} We investigate the influence of a shifting environment on the spreading of an invasive species through a model given by the diffusive logistic equation with a free boundary. When the environment is homogeneous and favourable, this model was first studied in Du and Lin \cite{DL}, where a spreading-vanishing dichotomy was established for the long-time dynamics of the species, and when spreading happens, it was shown that the species invades the new territory at some uniquely determined asymptotic speed $c_0>0$. Here we consider the situation that part of such an environment becomes unfavourable, and the unfavourable range of the environment moves into the favourable part  with speed $c>0$. We prove that when $c\geq c_0$, the  species always dies out in the long-run, but when $0<c<c_0$, the long-time behavior of the species is determined by a trichotomy described by
 (a) {\it vanishing}, (b) {\it borderline spreading}, or (c) {\it spreading}. 
If the initial population is writen in the form $u_0(x)=\sigma \phi(x)$ with $\phi$ fixed and $\sigma>0$ a parameter, then there exists $\sigma_0>0$ such that vanishing happens when $\sigma\in (0,\sigma_0)$, borderline spreading happens when $\sigma=\sigma_0$, and spreading happens when $\sigma>\sigma_0$.

\end{abstract}
\maketitle

\section{Introduction}
\setcounter{equation}{0}

The effect of climate change on the survival  of ecological  species has attracted a great deal of attention in recent years; see, for example, \cite{BDNZ, BR1, BR2, BN, LBSF}  and the references therein. To gain insights to this problem, some  simple mathematical models have been proposed and analyzed. One such model is  given by the Cauchy problem
\begin{equation}
\label{Cauchy}
u_t=du_{xx}+f(x-ct, u),\; x\in\bR^1, t>0;\; u(0,x)=u_0(x),\; x\in\bR^1,
\end{equation}
where $u(t,x)$ stands for the population density of the concerned species at time $t$ and spatial position $x$, with initial population $u_0(x)$. The climate change is incorperated in the function $f(x-ct,u)$, which describes  a changing environment  
that shifts with a certain speed $c>0$. In \cite{BDNZ, BR1, BR2, BN}, the situation where a shifting environment
with a favourable habitat range surrounded by unfavorable ones is investigated, and many interesting results are obtained, including useful criteria for long-time survival of the species.

In \cite{LBSF}, the influence of climate change on the spreading of an invasive species is studied, where  the problem is modeled  by \eqref{Cauchy} with a logistic type nonlinear term
\[
f(x-ct, u)=r(x-ct)u-u^2,
\]
and to represent a shifting environment, $r(\xi)$ is assumed to be a continuous increasing function with $r(\pm\infty)$ finite and $r(-\infty)<0<r(+\infty)$. So there is
a  $\xi_0\in\bR^1$ such that $r(\xi)\leq 0$ for $\xi<\xi_0$ and $r(\xi)>0$ for $\xi>\xi_0$, indicating that
the shifting range $\Omega_t^-:=\{x\in\bR^1:x<\xi_0+ct\}$  is unfavourable to the species, while  $\Omega_t^+:=\{x\in\bR^1:x>\xi_0+ct\}$ is favourable. 
The main result in \cite{LBSF} states that, if the environment shifting speed $c$ is strictly greater than $c^*:=2\sqrt{dr(+\infty)}$,
then the species will die out in the long run, while in the case $0<c<c^*$,  the species will survive and spread into new territory in the direction of the moving  environment with asymptotic speed $c^*$. More precisely, in the case $0<c<c^*$,  for any given small $\epsilon>0$,
\[
\lim_{t\to\infty}\left[\sup_{x\leq (c-\epsilon)t}u(t,x)\right]=0,\; \lim_{t\to\infty}\left[\sup_{x\geq (c^*+\epsilon)t}u(t,x)\right]=0,
\]
\[
\lim_{t\to\infty} \left[\sup_{(c+\epsilon)t\leq x\leq (c^*-\epsilon)t}|u(t,x)-r(+\infty)|\right]=0.
\]
Therefore, in the case $c<c^*$, the species will survive only inside the shifting range $S_t:=\{x\in\bR^1: ct<x<c^*t\}$ for large $t$.

In this paper we look at a similar problem to \cite{LBSF}, but use a free boundary to describe the spreading front of the species. As a matter of fact, we started working on the problem independently of \cite{LBSF}, and learned of \cite{LBSF} only after the first draft of our paper has been completed. It is a pleasing surprise to us that the nonlinear term in our model almost coincides with that used in \cite{LBSF}, which made the results arising from the two related models  readily comparable (see below, in particular Remark \ref{rm1.3} (ii)).

We now describe our model precisely.
Let $c>0$ be as before. We assume that $A(\xi)$ is a
Liptschitz continuous function on $\bR^1$ satisfying
 \begin{equation}\label{a}
 A(\xi)=
 \begin{cases}
 a_0, & \xi<-l_0,\\
 a, & \xi\geq 0,
 \end{cases}
\end{equation}
and $A(\xi)$ is strictly increasing over $[-l_0,0]$. Here $l_0$, $a_0$ and $a$ are constants, with $l_0>0$, $a_0\leq 0$ and $a>0$.

Our model is given by the following free boundary problem
\bes\left\{\begin{array}{ll}\medskip
\displaystyle u_t=du_{xx}+A(x-ct)u-bu^2,\ \ & t>0,\ \ 0<x<h(t),\\
\medskip\displaystyle u_x(t,0)=u(t,h(t))=0,\ \ & t>0,\\
\medskip\displaystyle h'(t)= -\mu u_x(t,h(t)),\ \ & t>0,\\
\displaystyle h(0)=h_0,\; u(0,x)=u_0(x),\ \ & 0\leq x\leq h_0,
\end{array}\right.\lbl{MP}
\ees
where $x=h(t)$ is the moving boundary to be determined, 
$h_0,\,\mu,d, b$ are positive constants, and
the initial function $u_0(x)$ satisfies
\begin{equation}\label{u_0}
 u_0\in C^2([0,h_0]),\ u_0'(0)=u_0(h_0)=0, \ u_0'(h_0)<0\ \text{and}\ u_0>0\ \text{in}\ [0,h_0).
\end{equation}
So in this model, the range of the species is the varying interval $[0, h(t)]$, and the species can invade the environment from the right end of the range ($x=h(t)$), with speed propotional to the population gradient $u_x$ there, while at the fixed boundary $x=0$, a no-flux boundary condition is assumed. The function $A(x-ct)$ represents the assumption that the unfavourable part of the environment is moving into the current and future habitat of the species at the speed $c$. We want to know the long-time dynamical behavior of $u(t,x)$.

In the case that $A(x-ct)$ is replaced by a positive constant $a$ (the same  as  in \eqref{a}), so the species is spreading in a favourable homogeneous environment, problem \eqref{MP} was studied in \cite{DL},
where a spreading-vanishing dichotomy was established. (See also \cite{DLou, DLZ} for a more systematic investigation of
similar free boundary models in homogeneous environment of one space dimension.) Moreover, in the case of spreading, it is shown in \cite{BDK, DMZ} that there exists $c_0=c_0(\mu)$ such that $h(t)-c_0t$ converges to some constant as $t\to+\infty$, and
\[
\lim_{t\to\infty}\left[\max_{0\leq x\leq h(t)}\big|u(t,x)-q_{c_0}(h(t)-x)\big|\right]=0,
\]
where $(c_0, q_{c_0}(\xi))$ is uniquely determined by 
\begin{equation}
\lbl{semi-wave}
\left\{
\begin{array}{l}
dq_{c_0}''-c_0q'_{c_0}+aq_{c_0}-bq_{c_0}^2=0,\; q_{c_0}>0  \mbox{ for } \xi>0;\\
q_{c_0}(0)=0,\; q_{c_0}(+\infty)=a/b,\; \mu q_{c_0}'(0)=c_0.
\end{array}\right.
\end{equation}
Moreover, $c_0$ is increasing in $\mu$ and $\lim_{\mu\to+\infty}c_0(\mu)=2\sqrt{ad}$ (see \cite{BDK}).
For comparison, let us remark that if  one takes $r(\xi)=A(\xi)$ in \cite{LBSF}, then the constant $c^*$ in the earlier discussions takes the value $2\sqrt{ad}$, which  is the asymptotic spreading speed of an invading species determined by  \eqref{Cauchy} with  the classical Fisher-KPP nonlinear term
$f=au-bu^2$ (see \cite{AW, Fisher, KPP}).

Using the techniques of \cite{DL}, it is easily seen that \eqref{MP} has a unique (classical) solution, which is defined for all $t>0$. The long-time dynamical behavior of the pair $(u(x,t), h(t))$ is given by the following two theorems.

\begin{theo}
\label{tri}
Let $(u,h)$ be the unique positive solution of \eqref{MP}. Suppose that $0<c<c_0$.
Then exactly one of the following happens:

\noindent
{\bf (i) Vanishing:}\, $\lim_{t\to\infty}h(t)=h_\infty<+\infty$ and
$$\lim\limits_{t\rightarrow\infty}\left[\max_{0\leq x\leq h(t)}u(t,x)\right]=0.$$

\noindent
{\bf (ii) Spreading:}\, 
$ \lim_{t\to\infty}h(t)/t=c_0,$
and for any small $\epsilon>0$,
\[
\lim_{t\to\infty}\left[\max_{(c+\epsilon)t\leq x\leq (1-\epsilon)h(t)}\big|u(t,x)-a/b\big|\right]=0,\;
\lim\limits_{t\rightarrow\infty}\left[\max_{0\leq x\leq (c-\epsilon)t}u(t,x)\right]=0.
\]

\noindent
{\bf (iii) Borderline Spreading:}\, 
$\lim_{t\rightarrow\infty}[h(t)-ct]=L_{*}$, and
\[
\lim_{t\rightarrow\infty}\left[\max_{0\leq x\leq h(t)}\big |u(t,x)-V_{*}(x-h(t)+L_*)\big|\right]=0,
\]
where $L_*>-l_0$ and $V_{*}(x)$ are uniquely determined by
\begin{equation}\left\{\begin{array}{ll}\medskip
\displaystyle dV_*''+cV_*'+A(x)V_*-bV_*^2=0,\; V_*>0\ \  \mbox{ for }  x\in (-\infty, L_{*}), \\
\medskip\displaystyle V_*(-\infty)=V_*(L_{*})=0,\;
\displaystyle -\mu V_*'(L_{*})=c.
\end{array}\right.
\lbl{L_*}
\end{equation}
\end{theo}

If the initial function in \eqref{MP} has the form $u_0(x)=\sigma \phi(x)$ with some fixed $\phi$ satisfying \eqref{u_0} and $\sigma>0$  a parameter,
then we will show that there exists $\sigma_0\in (0,+\infty]$ such that vanishing happens for $\sigma\in (0, \sigma_0)$,  spreading happens for $\sigma>\sigma_0$, and borderline spreading happens for $\sigma=\sigma_0$. Simple sufficient conditions can be found to guarantee
that $\sigma_0<+\infty$. The detailed statements of these results can be found in Section 4 below.

If $c\geq c_0$, we show that vanishing always happens, as indicated in the following result.

\begin{theo}
\lbl{c>c_0}

If $c\geq c_0$, then $\lim_{t\to\infty}h(t)=h_\infty<+\infty$ and
$$\lim\limits_{t\rightarrow\infty}\left[\max_{0\leq x\leq h(t)}u(t,x)\right]=0.$$

\end{theo}

\begin{remark}\lbl{rm1.3}{\rm
\begin{itemize}
\item[(i)] In Theorem \ref{tri}  case (ii), it is possible to use the techniques of \cite{DMZ} to show that
$\lim_{t\to\infty}[h(t)-c_0t]$ exists and is finite, and
\[
\lim_{t\to\infty}\left[\max_{(c+\epsilon)t\leq x\leq h(t)}\big|u(t,x)-q_{c_0}(h(t)-x)\big|\right]=0.
\]
To avoid the paper becoming too long, we have refrained from doing this here.
\item[(ii)] Compared with the phenomena revealed in \cite{LBSF} by using \eqref{Cauchy}, our Theorem \ref{tri} above captures some more varied long-time dynamical behaviors of the species for the case $0<c<c_0$. For the case $c\geq c_0$, our result here (Theorem \ref{c>c_0})
is paralelle to that for the  case
$c>c^*$ in \cite{LBSF}. 
\item[(iii)] It is interesting to note that in the case of favourable homogeneous environment considered in \cite{DL}, the long-time dynamical behavior of the species is governed by a spreading-vanishing dichotomy, while in the case of Theorem \ref{tri}, the long-time dynamical behavior is determined by a trichotomy. Theorems \ref{tri} and \ref{c>c_0} together clearly indicate that changing environments cause fundamental changes to the behavior of affected ecological species.
\item[(iv)] The trichotomy in Theorem \ref{tri} is similar in spirit to one of the main results in \cite{GLZ}, where 
a  free boundary problem with advection is considered in a homogeneous environment.
\end{itemize}}
\end{remark}

There are several recent related work on the free boundary model in spatially inhomogeneous environment (mostly for one space dimension or in a setting with spherical symmetry). In \cite{DLiang}, the case of periodic spatial environment is studied. Other types of heterogeneous spatial environments are considered in \cite{LLZ, mxW, ZX}. In \cite{DGP, PZ}, time-periodic environments are considered. See also the survey \cite{D2} for some further related research.

The rest of this paper is organized as follows. In Section 2, we give the existence and uniqueness result for \eqref{MP}, as well as results on several auxiliary elliptic problems, which will be used for proving the main results later. Section 3 is the main part of the paper, where we prove Theorem \ref{tri} through various comparison arguments, based on the construction of super-subsolutions,
and on suitable applications of a zero number result of Angenent \cite{Ang} in several key steps. In Section 4, we examine how the long-time dynamical behavior of \eqref{MP} changes as the initial function is varied. Section 5, the final section, constitutes the proof of Theorem \ref{c>c_0}.

\section{Preliminary results}

\subsection{ Existence and uniqueness} 
\setcounter{equation}{0}

The following local existence and uniqueness result can be proved by the contraction mapping theorem as in \cite{DL}.

\begin{theo}\lbl{local}{\rm (Local existence)}
 For any given $u_0$ satisfying \eqref{u_0} and any $\alpha\in(0,1)$, there is a $T>0$ such that problem \eqref{MP} admits a unique positive solution
 \[
  (u,h)\in C^{(1+\alpha)/2,1+\alpha}(D_T)\times C^{1+\alpha/2}([0,T]);
 \]
 moreover,
 \[
  \|u\|_{C^{(1+\alpha)/2,1+\alpha}(D_T)}+\|h\|_{C^{1+\alpha/2}([0,T])}\leq C,
 \]
 where $D_T=\{(t,x)\in \mathbb{R}^2: x\in [0,h(t)],\;t\in[0,T]\}$, $C$ and $T$ only depend on $h_0$, $\alpha$ and $\|u_0\|_{C^2([0,h_0])}$.
\end{theo}

 To show that the local solution obtained in Theorem \ref{local} can be extended to all $t>0$, as in \cite{DL},  the following estimates are useful.
\begin{lem}\lbl{estimate}
Let $(u,h)$ be a positive solution to problem \eqref{MP} defined for $t\in(0,T_0)$ for some $T_0\in(0,+\infty]$. Then there exist constants $C_1$ and $C_2$ independent of $T_0$ such that
\[
 0<u(t,x)\leq C_1, \ 0<h'(t)\leq C_2\ \ \text{for} \ 0\leq x<h(t)\ \mbox{and} \ t\in (0,T_0).
\]
\end{lem}

Using Theorem \ref{local} and Lemma \ref{estimate}, we can prove the following global existence result.

\begin{theo}\lbl{global}{\rm (Global existence)}
 The solution of problem \eqref{MP} is defined  for all $t\in (0,\infty)$.
\end{theo}

We omit the proofs of  these results as they are easy modifications of those  in \cite{DL}.

\subsection{Some auxiliary elliptic problems}

In this subsection,  we study several elliptic problems for later use. In particular we will prove the existence and uniqueness of $(L_*, V_*)$ appearing in 
\eqref{L_*}.

Let $c_0=c_0(\mu)$ and $ q_{c_0}(\xi)$ be given in \eqref{semi-wave}. We assume throughout this subsection that
\[
0<c<c_0.
\]

\begin{lem}\lbl{l-l}
Assume $C\in [0, 2\sqrt{ad}\,)$. Then for all
large $l>0$, the problem
\begin{equation}\label{w-l-l}
dw''+Cw'+aw-bw^2=0 \mbox{ for } -l<x<l,\;
w(-l)=w(l)=0
\end{equation}
admits a unique positive solution $w_l(x)$. Moreover,  $\lim_{l\rightarrow\infty} w_l(x)=\frac{a}{b}$ uniformly in any compact subset of $\bR^1$, and 
\[
\lim_{l\to+\infty}w_l'(l)=-C/\mu_C,
\]
where $\mu_C>0$ is  uniquely determined by  $c_0(\mu_C)=C$.
\end{lem}

\begin{proof}
We define
\[
 \lambda=\frac{C}{\sqrt{ad}}\ \ \text{and}\ \ \ v(x)=\frac{b}{a}e^{\frac{\lambda}{2}x}w(\sqrt{\frac{d}{a}}x).
\]
Then (\ref{w-l-l}) is changed to the equivalent problem
\begin{equation}\label{v-l-l}
\begin{cases}
-v''=(1-\frac{\lambda^2}{4})v-e^{-\frac{\lambda}{2}x}v^2\ \ & \mbox{ for } -\tilde{l}<x<\tilde{l},\\
v(-\tilde{l})=v(\tilde{l})=0,
\end{cases}
\end{equation}
with
\[
 \tilde{l}:= \sqrt{\frac{a}{d}}l.
\]
Due to $0\leq C<2\sqrt{ad}$, we have $1-{\lambda^2}/{4}>0$ and hence for all large $l$, by a well-known result on logistic type equations  (see, e.g., Theorem 5.1 in \cite{D}),  problem (\ref{v-l-l}) has a unique positive solution $v_l$, which in turn defines a unique positive solution $w_l$ for (\ref{w-l-l}).

Now we choose an increasing sequence $l_1<l_2<\cdots<l_n\rightarrow \infty$ with $l_1$  large enough so that $w_n:=w_{l_n}$ is defined for all $n\geq 1$. 
By the comparison principle (Lemma 2.1 in \cite{DM}), we have $w_n\leq w_{n+1}$ on $(-l_n,l_n)$. As any positive constant $M$ satisfying $M\geq {a}/{b}$ can be used as a supersolution of \eqref{w-l-l}, we see that $w_n\leq{a}/{b} $ for all $n$. Thus,  $w_\infty=\lim_{n\rightarrow\infty} w_n$ is well-defined on $\bR^1$. Furthermore, by the standard regularity considerations, we see that $w_n\to w_\infty$ in $C^2_{loc}(\bR^1)$ and $w_\infty$ satisfies
\begin{equation}\label{eq4.14}
-dw_{\infty}''-Cw_\infty'=aw_\infty-bw_\infty^2,\ \  x\in\bR^1.\\
\end{equation}
As $w_\infty\geq w_n>0$ on $(-l_n,l_n)$ for each $n$, we know that $w_\infty$ is a positive solution of (\ref{eq4.14}).

By Lemma 2.1 in \cite{DM},  we easily see $w_\infty(x)\geq w_n(x+x_0)$ on $(-l_n-x_0, l_n-x_0)$ for arbitrary $x_0$. 
Hence, for any $x\in\bR^1$, $w_\infty(x)\geq\max_{[-l_n,l_n]} w_n=\|w_n\|_\infty$ for all $n$. Let $n\rightarrow\infty$, we otain $w_\infty(x)\geq \lim_{n\rightarrow\infty}\|w_n\|_\infty=\|w_\infty\|_\infty$.
Hence $w_\infty(x)=\|w_\infty\|_\infty$   for $x\in\bR^1$. By (\ref{eq4.14}), $w_\infty$ equals to ${a}/{b}$.
So, $\lim_{l\rightarrow\infty}w_l(x)=a/b$ uniformly in any compact subset of $\bR^1$.

Define $U_l(x):=w_l(x+l)$.  From the proof of Proposition 2.1 in \cite{BDK} we know that 
\begin{equation}\label{eq4.35}
U_l(x)\rightarrow U_*(x) \ \ \text{in}\ C^1_{loc}((-\infty, 0])\ \ \ \ \text{as}\ \ l\rightarrow \infty,
\end{equation}
and $U_*$ is the unique positive solution of
\[
-dU''-CU'=aU-bU^2 \mbox{ for } x\in (-\infty, 0], \; U(0)=0.
\]
Moreover, if $c_0(\mu_C)=C$, then $U'_*(0)=-C/\mu_C$. Therefore
\[
\lim_{l\to+\infty} w_l'(l)=\lim_{l\to+\infty}U_l'(0)=U_*'(0)=-C/\mu_C.
\]
The proof is complete.
\end{proof}

Next we consider, for $l\geq 0$ and $L>-l$, the logistic type problem
\bes\left\{\begin{array}{ll}\medskip
\displaystyle dV''+cV'+A(x)V-bV^2=0,\ \ & -l<x<L, \\
\medskip\displaystyle V(-l)=V(L)=0.
 \end{array}\right.\lbl{V-l-L}
\ees
Let $\lambda_1[-l,L]$ denote the first eigenvalue of
\bess\left\{\begin{array}{ll}\medskip
\displaystyle -d\phi''-c\phi'-A(x)\phi=\lambda\phi,\ \ & -l<x<L, \\
\displaystyle \phi(-l)=\phi(L)=0.\\
 \end{array}\right.
\eess
Then \eqref{V-l-L} has a unique positive solution, which we denote by $V_{l, L}$, if and only if $\lambda_1[-l,L]<0$.

Since $A(x)=a$ for $x\geq 0$, and  $c\in (0, c_0)$, we see from the phase-plane analysis for case (iv) in Section 3.2 of \cite{GLZ} (note that the $c_0$ there is different from our $c_0$ here) that there exists a unique $L(0)>0$ such that \eqref{V-l-L} with $l=0$ and $L=L(0)$ has a unique positive solution $V_0$ satisfying
\[
-\mu V_0'(L(0))=c.
\]
We also note that the equation 
\[
dV''+cV'+A(x)V-bV^2=0
\] 
can be rewritten in the form
\[
-(d e^{\frac{c}{d}x}V')'=e^{\frac{c}{d}x}(A(x)V-bV^2),
\]
and hence the comparison principle in \cite{DM} can be applied directly to this equation.

\begin{lem}\lbl{lem3}
$(i)$\, For each $l>0$, there is a unique $L(l)>-l$ such that
\eqref{V-l-L} with $L=L(l)$
has a unique positive solution $V_{l}$ satisfying $-\mu V_l'(L(l))=c;$

$(ii)$\, The function $l\to L(l)$ is decreasing, and $L_{*}:=\lim_{l\rightarrow\infty}L(l)>-l_0;$

$(iii)$\,  $V_*(x):=\lim_{l\to\infty}V_l(x)$ exists and it is the unique positive solution of
\bes\left\{\begin{array}{l}\medskip
\displaystyle dV''+cV'+A(x)V-bV^2=0 \mbox{ for } -\infty<x<L_{*}, \\
\displaystyle V(-\infty)=V(L_{*})=0,\; -\mu V'(L_{*})=c.
\end{array}\right.\lbl{V-L_*}
\ees
\end{lem}
\begin{proof}\,
For any $l>0$, since $\lambda_1[-l,L(0)]<\lambda_1[0, L(0)]<0$, \eqref{V-l-L} with $L=L(0)$ has a unique positive solution  $V_{l,L(0)}$.
By the comparison principle and the Hopf boundary lemma, we  have
\[
V_0(x)<V_{l,L(0)}(x)\ \mbox{for}\ x\in[0,L(0)) \mbox{ and } V'_{l, L(0))}(L(0))<V_0'(L(0)).
\]
Hence
\[
-\mu V_{l,L(0)}'(L(0))>c.
\]

Since $A(x-l)\leq,\not\equiv A(x)$ for $x\in [0, L(0)]$, we have 
\[
\lambda_1[-l, L(0)-l]>\lambda_1[0, L(0)].
\]
If $\lambda_1[-l, L(0)-l]<0$, then \eqref{V-l-L}  with $L=L(0)-l$ has a unique positive solution $V_{l, L(0)-l}$, and we can use the comparison principle and the Hopf boundary lemma to deduce
\[
V_{l, L(0)-l}(x-l)<V_0(x) \mbox{ for } x\in (0, L(0)),\; V'_{l, L(0)-l}(L(0)-l)>V_0'(L(0)).
\]
Hence
\[
-\mu V'_{l, L(0)-l}(L(0)-l)<c.
\]

If $\lambda_1[-l, L(0)-l]\geq 0$, then we can find a unique $L'\in[ L(0)-l, L(0))$ such that $\lambda_1[-l, L']=0$.
Thus for $L\in (L', L(0))$ we have $\lambda_1[-l, L]<0$ and \eqref{V-l-L} has a unique positive solution $V_{l,L}$.
Moreover, a well-known property of the logistic type equation indicates that $\lim_{L\to L'}\|V_{l,L}\|_{C^2[-l, L]}= 0$.
Hence $-\mu V'_{l, L}(L)<c$ for $L>L'$ but close to $L'$.

Therefore whether $\lambda_1[-l, L(0)-l]< 0$ or $\lambda_1[-l, L(0)-l]\geq 0$, we can always find some $L\in (-l, L(0))$ such that
$-\mu V_{l, L}'(L)<c$. Since $-\mu V_{l,L(0)}'(L(0))>c$, by the continuous dependence of $V'_{l, L}(L)$ on $L$, there exists
$L(l)\in (L, L(0))$ such that 
\[
-\mu V'_{l, L(l)}(L(l))=c.
\]
Moreover, for $L'<L_1<L_2<L(0)$, we can compare $V_{l,L_1}(x)$ with $V_{l, L_2}(x+L_2-L_1)$ over $x\in [-l, L_1]$, and use the comparison principle and Hopf Lemma, to deduce that $ V'_{l,L_1}(L_1)>V'_{l, L_2}(L_2)$.
This implies that $L(l)$ is uniquely determined. The proof of conclusion (i) is now complete.

We next prove the first part of (ii). For convenience, we denote $V_{l}=V_{l,L(l)}$. 
Arguing indirectly, we assume that there are
$l_1>l_2\geq 0$ satisfying $L_1:=L(l_1)\geq L_2:=L(l_2)$. Denote $V_1(x)=V_{l_1}(x+L_1-L_2)$, then
\bess\left\{\begin{array}{ll}\medskip
\displaystyle -dV_1''-cV_1'\geq A(x)V_1-bV_1^2,\ \ & -l_2<x<L_2, \\
\displaystyle V_1(-l_2)>0, V_1(L_2)=0.\\
\end{array}\right.
\eess
By the comparison principle we have
\[V_1(x)>V_{l_2}(x) \ \mbox{for}\ x\in[-l_2,L_2).
\]
By the Hopf lemma,
$V_1'(L_2)<V_{l_2}'(L_2)=-c/\mu$, which contradicts $V_1'(L_2)=V'_{l_1}(L_1)=-c/\mu$. The first part of conclusion (ii) is now proved.

To prove the second part of (ii) and conclusion (iii),
 let $\{l_n\}$ be an increasing sequence that converges to $\infty$. Denote $L_n:=L(l_n)$ and $W_n(x):=V_{l_n}(x+L_n)$;
 then
\bess\left\{\begin{array}{ll}\medskip
\displaystyle dW_n''+cW_n'+A(x+L_n)W_n-bW_n^2=0,\ \ & -l_n-L_n<x<0, \\
\medskip\displaystyle W_n(-l_n-L_n)=W_n(0)=0,\;
 W_n'(0)=-c/\mu.
\end{array}\right.
\eess

We first observe that for any $l>0$, $L(l)>-l_0$. Otherwise $L(l)\leq -l_0$ for some $l>0$. It follows that $A(x)\equiv a_0\leq 0$ for $x\in [-l, L(l)]$. If the maximum of $V_l$ in $[-l, L(l)]$ is attained at $x_0\in(-l, L(l))$, then we arrive at the following contradiction:
\[
0=dV_l''(x_0)+c V_l'(x_0)+a_0 V_l(x_0)-b V_l^2(x_0)\leq -b V_l^2(x_0)<0.
\]
Hence we always have $L(l)>-l_0$ and so $L_*:=\lim_{n\to\infty}L_n\geq -l_0$.

Since $0\leq W_n(x)\leq a/b$ for all $x\in[-l_n-L_n,0]$, by regularity arguments of elliptic equations and a standard diagonal process, there is a subsequence of $\{W_n\}$, for convenience, still denoted by itself, such that
\[
W_n\rightarrow W_* \ \mbox{in}\ C^1_{loc}((-\infty,0]),\ \ a/b\geq W_*(x)\geq 0\ \mbox{for}\ x<0
\]
and
\bess\left\{\begin{array}{ll}\medskip
\displaystyle dW_*''+cW_*'+A(x+L_{*})W_*-bW_*^2=0 \mbox{ for }  x<0, \\
\displaystyle W_*'(0)=-c/\mu,\ \ W_*(0)=0.
\end{array}\right.
\eess
Since $W_*'(0)=-c/\mu<0$, by the strong maximum principle, we necessarily
have $W_*(x)>0$ for $x\in (-\infty, 0)$.
Define $V_*(x):=W_*(x-L_{*})$; then $V_*(x)$ satisfies
\bess\left\{\begin{array}{ll}\medskip
\displaystyle dV_*''+cV_*'+A(x)V_*-bV_*^2=0,\; V_*>0 \mbox{ for }  x<L_{*}, \\
\displaystyle V_*'(L_{*})=-c/\mu,\ \ V_*(L_{*})=0.
\end{array}\right.
\eess

We claim that $V_*(x)\rightarrow 0$ as $x\rightarrow-\infty$.
Obviously, by regularity arguments of elliptic equations and the definition of $A$, $V_*\in C^2((-\infty,L_*])$ and
\bes\lbl{x<-l_0}
dV_*''+cV'_*\geq bV_*^2 \ \ \mbox{for}\  x\leq -l_0.
\ees
Therefore
\[
(e^{\frac{c}{d}x}V_*')'\geq \frac{b}{d}e^{\frac{c}{d}x}V^2_*>0\ \ \mbox{for}\ x\leq -l_0.
\]
Hence, $e^{\frac{c}{d}x}V_*'(x)$ is an increasing function in $(-\infty,-l_0]$. Since $V_*$ is a bounded function, there exists $\{x_n\}$ satisfying $x_n\rightarrow -\infty$ such that
\[
V_*'(x_n)\rightarrow 0 \ \mbox{as}\ n\rightarrow\infty.
\]
It follows that
\[
e^{\frac{c}{d}x}V_*'(x)>\lim_{n\rightarrow\infty}e^{\frac{c}{d}x_n}V_*'(x_n)=0\ \mbox{ for every }  x\in(-\infty,-l_0].
\]
Therefore, we have
$V_*'(x)>0$ in $(-\infty,-l_0]$, and  $V_*(x)$ is increasing in $(-\infty,-l_0]$.
This fact and $V_*(L_*)=0$ clearly imply $L_*>-l_0$.

Denote $\tilde m:=\lim_{x\rightarrow-\infty}V_*(x)$; then clearly $\tilde m\geq 0$.
If $\tilde m>0$, then 
\[
\lim_{x\to-\infty}[dV_*''(x)+cV_*'(x)]=\alpha:=-a_0\tilde m+b\tilde m^2>0,
\]
from which we immediately obtain $V_*(x)\to-\infty$ as $x\to-\infty$. 
This contradiction shows that $\tilde m=0$. 

Finally we note that due to conclusion (ii), $L_*=\lim_{l\to+\infty}L(l)$ is uniquely determined. The uniqueness of $V_*$ follows from the uniqueness of initial value problems of second order ODEs, since
$V_*$ can be viewed as the unique solution of the initial value problem
\[dV''+cV'=bV^2-A(x)V,\; V(L_*)=0, \; V'(L_*)=-c/\mu.
\]
The proof is complete.
\end{proof}

\begin{remark}\lbl{l}{\rm Let us observe that for any $l\leq 0$, due to $A(\xi)=a$ for $\xi\geq 0$, $V_l(x):=V_0(x+l)$ is the unique positive solution of \eqref{V-l-L} with $L=L(l):=L(0)-l$ that satisfies $-\mu V_{l}'(L)=c$. From the proof of Lemma \ref{lem3}, it is easily seen that $L(l)\to L(0)$ as $l\to 0^+$. Therefore $L(l)$ is a continuous and strictly decreasing function of $l$ for $l\in\bR^1$, with $L(+\infty)=L_*,\; L(-\infty)=+\infty$.
}
\end{remark}

\begin{remark}\lbl{M} {\rm
Consider the following problem
\bes\left\{\begin{array}{ll}\medskip
\displaystyle dW''+cW'+A(x)W-bW^2=0,\ \ & -l<x<L_{*}, \\
\displaystyle W(-l)=M,\ \ W(L_{*})=0,\\
\end{array}\right.\lbl{M-L-*}
\ees
where $M=\max\{\|u_0\|_\infty,a/b\}$.
By a simple super-subsolution argument, and the comparison principle {\rm (\cite{DM})},
(\ref{M-L-*}) has a unique positive solution $W_l(x)$, $M\geq W_l(x)>V_*(x)$ in $[-l, L_*)$, and $W_l$ is decreasing in $l$.
By the regularity theory of elliptic equations, there exists $\alpha\in (0,1)$ such that, as $l\to+\infty$, 
\[
W_{l}\rightarrow\ W_*\ \mbox{in}\ C_{loc}^{1+\alpha}((-\infty,L_{*}]),
\]
and $W_*$ satisfies
\[
dW_*''+cW_*'+A(x)W_*-b W_*^2=0\ \mbox{in}\ (-\infty,L_{*}]; \ W_*(L_{*})=0.
\]
Similar to the proof of Lemma \ref{lem3}, we can show
$W_*(x)\rightarrow 0$ as $x\rightarrow-\infty$. We may then argue as in the proof of the comparison principle in  \cite{DM} to
deduce $W_*\equiv V_*$. $($So the uniqueness of $V_*$ can also be deduced from $V_*(-\infty)=0$, instead of using $V_*'(L_*)=-c/\mu.)$
}
\end{remark}

The following result will be useful later in the paper.

\begin{lem}\lbl{lem4}
For any given $L<L_{*}$ and $-l<L$, let $W_{l,L}$ denote the unique positive solution of \eqref{M-L-*} with $L_{*}$ replaced by $L$. Then for all sufficiently large $l$,
\[
-\mu W_{l,L}'(L)<c.
\]
\end{lem}
\begin{proof}\,
If the conclusion is not ture, then there is $\{l_n\}$ converging to $\infty$, such that
\[
-\mu W_{l_n,L}'(L)\geq c.
\]
By regularity arguments of elliptic equations, there exist $\alpha\in (0,1)$ and a subsequence of $\{l_n\}$, for convenience still denoted it by itself, such that as $n\to\infty$,
\[
W_{l_n,L}\rightarrow\ W_L\ \mbox{in}\ C_{loc}^{1+\alpha}((-\infty,L]).
\]
Then, it is easily seen that
\bes\lbl{W_L-equation}\left\{\begin{array}{ll}\medskip
\displaystyle dW_L''+cW_L'+A(x)W_L-bW_L^2=0,\ \ & -\infty<x<L, \\
\displaystyle W_L(-\infty)=W_L(L)=0,
\end{array}\right.
\ees
and
$
-\mu W_{L}'(L)\geq c$. 

Let $\tilde V_*(x)=V_{*}(x-L+L_{*})$; then from $A(x-L+L_*)\geq A(x)$ we obtain
\[
d\tilde V_*''+c\tilde V_*'+A(x)\tilde V_*-b\tilde V_*^2\leq 0\ \mbox{in}\ (-\infty,L).
\]
Since $\tilde V_*(-\infty)=W_L(-\infty)=0$,  as in the proof of   the comparison principle in \cite{DM}, we 
can show  $\tilde V_*\geq W_{L}$ in $(-\infty, L]$, and hence
\[
-\mu W_L'(L)\leq -\mu \tilde V_*'(L)=-\mu V_*'(L_*)=c.
\]
 So, $-\mu W_L'(L)=c$ holds. As before, we can show that $W_L'(x)>0$ in $(-\infty, -l_0]$, which together with $W_L(L)=0$ implies $L>-l_0$. Then $\tilde V_{*}$ is a strict supersolution of (\ref{W_L-equation}).
By the comparison principle, strong maximum principle and Hopf Lemma, we obtain
$\tilde V_*>W_L$ in $(-\infty, L)$ and $W_L'(L)>\tilde V_*'(L)=V_*'(L_*)$. Hence
\[
-\frac{c}{\mu}=W_L'(L)>V_{*}'(L_{*})=-\frac{c}{\mu}.
\]
This contradiction finishes the proof.
\end{proof}

\section{The Trichotomy}
\setcounter{equation}{0}

In this section, we will prove the following trichotomy result, which clearly implies Theorem \ref{tri}.
\begin{theo}\lbl{three}
Suppose that $c\in(0, c_0)$ and $(u,h)$ is the unique  solution of \eqref{MP}. Then
\begin{itemize}
  \item[(i)] vanishing happens  if $\limsup_{t\rightarrow\infty}[h(t)-ct]<L_{*};$
  \item[(ii)] borderline spreading happens  if $\limsup_{t\rightarrow\infty}[h(t)-ct]=L_{*};$
  \item[(iii)] spreading happens if  $\limsup_{t\rightarrow\infty}[h(t)-ct]>L_{*}.$
\end{itemize}
\end{theo}

The proof of this theorem will take up the rest of this section, which is divided into four subsections.
 Unless otherwise specified, throughout this section, we always assume that 
\[
0<c<c_0.
\]

\subsection{Properties of $h(t)$}
In this subsection, we prove some important properties of $h(t)$, which form the conner stones in the proof of Theorem \ref{tri}. Our arguments here are based on various comparison techniques, and following \cite{DLZ} and \cite{GLZ}, in several key steps we will
make use of some zero number results derived from Angenent \cite{Ang}.

\begin{lem}\lbl{prop-infty}
If $\limsup\limits_{t\rightarrow\infty}[h(t)-ct]=\infty$, then $\lim\limits_{t\rightarrow\infty}[h(t)-ct]=\infty$.
\end{lem}
\begin{proof}\,
 For arbitrarily  given $\tilde l>h_0$, define
\[
I(t)=[z_1(t),z_2(t)]:=[ct+\tilde l,ct+L(0)+\tilde l]
\]
and
\[
w(t,x):=V_0(x-ct-\tilde l),\ \ t>0,\ \ x\in I(t),
\]
where $V_0$ is the unique positive solution of (\ref{V-l-L}) with $l=0$ and $L=L(0)$.
Obviously,
\bes\left\{\begin{array}{ll}\medskip
\displaystyle w_t=dw_{xx}+aw-bw^2,\ \ & t>0,\ \ ct+\tilde l<x<ct+\tilde l+L(0),\\
\medskip\displaystyle w(t,z_1(t))=w(t,z_2(t))=0,\ \ & t>0,\\
\medskip\displaystyle w(0,x)=V_0(x-\tilde l),\ \ & \tilde l\leq x\leq \tilde l+L(0),\\
\displaystyle -\mu w_x(t,z_2(t))=c,\ \ & t>0.
\end{array}\right.\lbl{w(t,x)}
\ees
Since $\limsup_{t\rightarrow\infty}[h(t)-ct]=\infty$, and $h(0)=h_0<\tilde l=z_1(0)<z_2(0)$, and $h(t)$ is continuous, 
 we can find $t>0$ such that $h(t)=z_2(t)$.
Denote the smallest such $t$ by $t_1$. There must exist $t_2\in(0,t_1)$ such that
$$h(t_2)=z_1(t_2)\ \mbox{and} \ z_1(t)<h(t)<z_2(t) \ \mbox{when} \ t\in(t_2,t_1).$$
Obviously,
$$h'(t_1)\geq z_2'(t_1)=c.$$

Denote $\eta(t,x):=u(t,x)-w(t,x)$, $J(t):=[z_1(t),h(t)]$ and let $\mathcal{Z}_{J(t)}(\eta(t,\cdot))$ be the number of zeroes of $\eta(t,\cdot)$ in $J(t)$. Since $\eta(t,z_1(t))=u(t,z_1(t))>0$ and $\eta(t,h(t))=-w(t,h(t))<0$ for $t\in(t_2,t_1)$, $\mathcal{Z}_{J(t)}(\eta(t,\cdot))\geq 1$, and the zero number result of Angenent \cite{Ang} (see Lemma 2.2 in \cite{DLZ} for a convenient version) can be applied.

For all $t>t_2$ that is close to $t_2$, by the Hopf lemma and continuity,
$u_x(t,x)<0$ and $w_x(t,x)>0$ for $x\in J(t)$. This implies $\eta_x(t,x)<0$ for such $t$ and $x\in J(t)$. Therefore, for such $t$, $\eta(t,\cdot)$ has only one zero in $J(t)$, and it is a nondegenerate zero. Since by the zero number result
$\mathcal{Z}_{J(t)}(\eta(t,\cdot))$ is nonincreasing in $t$ for $t\in(t_2,t_1)$, the only possible case is that $\eta(t,\cdot)$ has exactly one zero for every $t\in (t_2, t_1)$;
the zero number result further implies that this zero is nondegenerate. So, the zero of $\eta(t,\cdot)$ in $(t_2,t_1)$ can be expressed as a smooth curve $x=z(t)$.

We claim that $z(t)$ converges to $h(t_1)$ when $t$ increases to $t_1$. Clearly,
$$z_1(t_1)\leq x_*:=\liminf_{t\rightarrow t_1^-}z(t)\ \mbox{and}\ x^*:=\limsup_{t\rightarrow t_1^-}z(t)\leq h(t_1).$$
If $x_*<x^*$, then it is easily seen that $\eta(t_1,x)\equiv 0$ for $x\in[x_*,x^*]$. We may apply Theorem 2 of \cite{F} to $\eta$ over the region
$[t_1-\epsilon,t_1]\times [z_1(t_1+\epsilon),h(t_1-\epsilon)]$, with $\epsilon>0$ sufficiently small, to conclude that $\eta(t_1,x)\equiv 0$ for
$x\in[z_1(t_1+\epsilon),h(t_1-\epsilon)]$. Letting $\epsilon\rightarrow 0$, we have
$\eta(t_1,x)\equiv 0$ for $x\in[z_1(t_1),h(t_1)]$, which contradicts $\eta(t_1,z_1(t_1))>0$.
Therefore $z(t_1):=\lim_{t\rightarrow t_1^-}z(t)$ exists.

By way of contradiction, we assume $z(t_1)<h(t_1)=ct_1+L(0)+\tilde l$.
Consider $\eta(t,x)$ in the domain
$$\{(t,x):z(t)<x<h(t),\ t_2<t\leq t_1\}.$$
By the maximum principle and Hopf lemma, we have
$$\eta(t_1,x)<0 \ \mbox{for} \ x\in(z(t_1),h(t_1)),\ \ \eta(t_1,h(t_1))=0 \ \mbox{and}\ \eta_x(t_1,h(t_1))>0.$$
The last inequality implies that
$$
-\mu u_x(t_1,h(t_1))<-\mu w_x(t_1,ct_1+L(0)+\tilde l)=-\mu V_0'(L(0)),
$$
that is
$$
h'(t_1)<c,
$$
which is in contradiction with our earlier inequality $h'(t_1)\geq c$.
This proves  $z(t_1)=h(t_1)$.

 We may now use the maximum principle to $\eta(t,x)$ over
$\{(t,x): t_2<t\leq t_1, z_1(t)<x<z(t)\}$ to deduce
$$u(t_1,x)>w(t_1,x)\ \ \mbox{for} \ x\in[ct_1+\tilde l,ct_1+\tilde l+L(0)).$$
Hence we can easily see,  by the comparison principle for free boundary problems (see \cite{DL}), that
$$u(t+t_1,x)\geq w(t+t_1,x)\ \ \mbox{for}\ t>0\ \mbox{and} \; x\in (z_1(t+t_1), z_2(t+t_1)),$$
and
$$h(t+t_1)\geq z_2(t+t_1)\ \ \mbox{for all}\ t>0.$$
So for any $t>t_1$,
$$h(t)-ct\geq \tilde l+L(0).$$
Since $\tilde l$ can be arbitrarily large, this implies $\lim_{t\rightarrow\infty}[h(t)-ct]=\infty$.
\end{proof}

\begin{lem}\lbl{speed}
If $\limsup\limits_{t\rightarrow\infty}[h(t)-ct]=\infty$, then $\lim\limits_{t\rightarrow\infty}\frac{h(t)}{t}=c_0$.
\end{lem}
\begin{proof}\,
Since $A(x-ct)\leq a$, by the comparison principle and estimate of spreading speed in \cite{DL}, we have
$$
\limsup_{t\rightarrow\infty}\frac{h(t)}{t}\leq c_0.
$$
Therefore, for our aim here, it suffices to show that
\bes
\lbl{inf}
\liminf_{t\rightarrow\infty}\frac{h(t)}{t}\geq\tilde c\;\; \ \ \forall\tilde c\in(c, c_0).
\ees

We now set to prove (\ref{inf}). For any given $\tilde c\in(c, c_0)$, as in Section 2 above, by the phase-plane analysis in \cite{GLZ}, there exists a unique pair $(\tilde L,\tilde V(x))$  with $\tilde V(x)>0$ in $(0, \tilde L)$
such that \eqref{V-l-L} is satisfied with $l=0$ and $(c,L,V)=(\tilde c, \tilde L,
\tilde V)$.  By the strong maximum principle, there exists $\epsilon>0$ such that
\bes
\lbl{ab-epsilon} 
\tilde V(x)\leq \frac{a}{b}-\epsilon\ \ \mbox{for} \ 0 \leq x\leq \tilde L.
\ees
By Lemma \ref{l-l}, for all sufficiently large $l$,
\bess
\left\{\begin{array}{ll}\medskip
\displaystyle dw''+cw'+aw-bw^2=0,\ \ &\ \ -l<x<l,\\
\displaystyle w(-l)=w(l)=0
\end{array}\right.
\eess
has a unique positive solution $w_l$, and  $w_l\rightarrow a/b$ uniformly in any compact subset of $(-\infty,\infty)$. So, there exists $\tilde l>\tilde L$ such that 
\bess
w_{\tilde l}(x)>\frac{a}{b}-\epsilon/2\ \ \mbox{for} \ -\frac{\tilde L}{2}\leq x\leq\frac{\tilde L}{2}.
\eess
Denote $\psi_*(x)=w_{\tilde l}(x-\tilde l)$; then
\bess
\left\{\begin{array}{ll}\medskip
\displaystyle d\psi''_*+c\psi'_*+a\psi_*-b\psi_*^2=0,\ \ &\ \ 0<x<2\tilde l,\\
\displaystyle \psi_*(0)=\psi_*(2\tilde l)=0
\end{array}\right.
\eess
and
\bes\lbl{psi_*}
\psi_*(x)>\frac{a}{b}-\epsilon/2 \ \ \mbox{for} \ \tilde l-\frac{\tilde L}{2}\leq x\leq \tilde l+\frac{\tilde L}{2}.
\ees
For any $\psi_0\in C^1([0,2\tilde l\,])$ satisfying $\psi_0(x)>0$ for $x\in(0,2\tilde l\,)$, and $\psi_0(0)=\psi_0(2\tilde l\,)=0$,  the 
auxiliary initial boundary value problem
\bes\lbl{psi-equation}\left\{\begin{array}{ll}\medskip
\displaystyle \psi_t=d\psi_{xx}+c\psi_x+a\psi-b\psi^2,\ \ & t>0,\ \ 0<x<2\tilde l,\\
\medskip\displaystyle \psi(t,0)=\psi(t,2\tilde l\,)=0,\ \ & t>0,\\
\displaystyle \psi(0,x)=\psi_0(x),\ \ & 0\leq x\leq 2\tilde l
\end{array}\right.
\ees
has a unique positive solution $\psi(t,x;\psi_0)$, and it is well known that
$$\psi(t,x;\psi_0)\rightarrow\psi_*(x)\ \mbox{uniformly  as} \ t\rightarrow\infty.$$
By (\ref{psi_*}), there exists $T=T(\psi_0)$ such that when $t\geq T$,
\bes
\lbl{psi-epsilon}\psi(t,x;\psi_0)>\frac{a}{b}-\epsilon \ \ \mbox{for} \ \tilde l-\frac{\tilde L}{2}\leq x\leq \tilde l+\frac{\tilde L}{2}.
\ees

Now, denote
$v(t,x):=u(t,x+ct)$ and $g(t):=h(t)-ct$. Due to $\limsup_{t\rightarrow\infty}[h(t)-ct]=\infty$ and 
Lemma \ref{prop-infty}, there exists $T_1>0$ such that
$g(t)>2\tilde l$ for all $t\geq T_1$.
We also have
\bess\left\{\begin{array}{ll}\medskip
\displaystyle v_t=dv_{xx}+cv_x+av-bv^2,\ \ & t>T_1,\ \ 0<x<g(t),\\
\medskip\displaystyle v(t,0)>0, \ \ v(t,g(t))=0,\ \ & t>T_1,\\
\displaystyle v(T_1,x)=u(T_1,x+cT_1),\ \ & 0\leq x\leq g(T_1).
\end{array}\right.
\eess
Therefore if we have chosen $\psi_0$ in (\ref{psi-equation}) satisfying $0<\psi_0(x)\leq u(T_1,x+cT_1)$ for $0<x<2\tilde l$, then by the comparison principle,
$$\psi(t,x;\psi_0)<v(t+T_1,x)\ \ \mbox{for} \ t>0\ \mbox{and} \ 0<x<2\tilde l.$$
By virtue of (\ref{psi-epsilon}), we have
\bes\lbl{v-epsilon}
v(T_1+T,x)>\frac{a}{b}-\epsilon \ \ \mbox{when} \ \tilde l-\frac{\tilde L}{2}<x<\tilde l+\frac{\tilde L}{2}.
\ees
Denote $T_0=T+T_1$, and we obtain, from (\ref{v-epsilon}), 
$$u(T_0,x)>\frac{a}{b}-\epsilon \ \ \mbox{for} \ \tilde l+cT_0-\frac{\tilde L}{2}<x<\tilde l+cT_0+\frac{\tilde L}{2}.$$
Now we set
$$\underline u(t,x):=\tilde V(x-\tilde ct-\tilde l-cT_0+\frac{\tilde L}{2}),$$
$$\xi_1(t):=\tilde ct+\tilde l+cT_0-\frac{\tilde L}{2},\ \ \xi_2(t):=\tilde ct+\tilde l+cT_0+\frac{\tilde L}{2}.$$
Clearly
$$\underline u_t=d\underline u_{xx}+a\underline u-b\underline u^2\ \ \mbox{for} \ t>0,\ \xi_1(t)<x<\xi_2(t),$$
$$\underline u(t,\xi_1(t))=\underline u(t,\xi_2(t))=0,$$
$$-\mu\underline u_x(t,\xi_2(t))=\tilde c=\xi_2'(t),$$
$$\xi_1(0)>cT_0,\ \ \xi_2(0)<h(T_0),$$
and
$$\underline u(0,x)=\tilde V(x-\tilde l-cT_0+\frac{\tilde L}{2})\leq \frac ab-\epsilon<u(T_0,x)\ \ \mbox{when}\ \xi_1(0)\leq x\leq\xi_2(0).$$
By the comparison principle,
\bes\lbl{u>}
u(t+T_0,x)\geq \underline u(t,x)\ \mbox{for}\ t\geq 0, \ \ x\in[\xi_1(t),\xi_2(t)], 
\ees
and
\[
h(t+T_0)\geq \xi_2(t)=\tilde ct+\tilde l+cT_0+\frac{\tilde L}{2} \ \mbox{for all} \ t>0,
\]
which implies (\ref{inf}).
\end{proof}

\begin{lem}\lbl{prop3}\,
If $\lim\sup_{t\rightarrow\infty}[h(t)-ct]=H^*$ and $L_{*}<H^*<\infty$, then $\lim\limits_{t\rightarrow\infty}[h(t)-ct]=H^*$.
\end{lem}
\begin{proof}\,
It suffices to show that for any $\tau\in(0,H^*-L_{*})$, $h(t)-ct\geq H^*-\tau$ for all large  $t$.
We now fix an arbitrary $\tau\in (0,H^*-L_{*})$.
Since $\limsup_{t\rightarrow\infty}[h(t)-ct]=H^*$, there exists $T_0>0$ such that $$h(T_0)-cT_0>H^*-\tau.$$
Due to $H^*-\tau>L_*$, by Remark \ref{l} and Lemma \ref{lem3}, there exists a unique pair $(l_*,V_{l_*})$ such that $L(l_*)=H^*-\tau$ and
\bess\left\{\begin{array}{ll}\medskip
\displaystyle dV_{l_*}''+cV_{l_*}'+A(x)V_{l_*}-bV_{l_*}^2=0,\ V_{l_*}>0 \mbox{ for } x\in(-l_*,H^*-\tau),\\
\medskip\displaystyle V_{l_*}(-l_*)=V_{l_*}(H^*-\tau)=0,\
\displaystyle -\mu V_{l_*}'(H^*-\tau)=c.
\end{array}\right.
\eess

\noindent
{\bf Claim 1:}\, There exists $T\geq T_0$ such that $h(t)-ct>-l_*$ when $t\geq T$. 

Otherwise, due to $\limsup_{t\rightarrow\infty}[h(t)-ct]=H^*$, there exist $t_2>t_1>T$ such that
$$ h(t_1)-ct_1=-l_*,\ h(t_2)-ct_2=H^*-\tau$$ and
$$-l_*<h(t)-ct<H^*-\tau\ \ \mbox{for}\ t\in(t_1,t_2).$$
Set
$v(t,x):=u(t,x+ct)$.
Then,
\[
v_t=dv_{xx}+cv_x+A(x)v-bv^2 \ \mbox{for}\ t>0, -ct<x<h(t)-ct.
\]
Similar to the proof of Lemma \ref{prop-infty}, by using the zero number argument and suitable comparison principles, we
can prove 
\[
v(t_2,x)>V_{l_*}(x)\ \ \mbox{for}\ -l_*\leq x<H^*-\tau,
\]
that is
\[
u(t_2,x+ct_2)>V_{l_*}(x) \ \ \mbox{for}\ -l_*\leq x<H^*-\tau.
\]
To stress the dependence of $l_*$ on $c$, we now write $l_*=l_{*,c}$,  with $L(l_{*,c})$, $V_{l_{*,c}}$  understood accordingly. 
By the Hopf lemma and continuity, there exist small $\delta>0$ such that
\[
u(t_2,x+ct_2)>V_{l_{*, c+\delta}}(x)\ \ \mbox{for} \ x\in[-l_{*, c+\delta},H^*-\tau).
\]
We now define
$$\underline u(t,x):=V_{l_{*, c+\delta}}(x-ct_2-(c+\delta)t),$$
\[
\xi_1(t)=(c+\delta)t+ct_2-l_{*, c+\delta},
\]
\[\xi_2(t)=(c+\delta)t+ct_2+H^*-\tau.
\]
Then for $t>0$ and $x\in(\xi_1(t),\xi_2(t))$,
\bess
\underline u_t&=&d\underline u_{xx}+A(x-ct_2-(c+\delta)t)\underline u-b\underline u^2\\
&\leq& d\underline u_{xx}+A(x-c(t+t_2))\underline u-b\underline u^2,
\eess
\[
\underline u(t,\xi_1(t))=\underline u(t,\xi_2(t))=0,
\]
\[
-\mu\underline u_x(t,\xi_2(t))=c+\delta=\xi_2'(t)
\]
and
\[
\underline u(0,x)\leq u(t_2,x)\ \mbox{for}\ x\in[\xi_1(0),\xi_2(0)].
\]
Applying the comparison principle to $(u(x,t_2+t), h(t_2+t))$ and $(\underline u(x,t), \xi_2(t))$ over $\{(x,t): x\in [\xi_1(t),\xi_2(t)], t>0\}$, we obtain, for $t>0$,
\[
u(t_2+t,x)\geq \underline u(t,x)\ \mbox{for}\  x\in [\xi_1(t),\xi_2(t)]
\]
and
\[
h(t+t_2)\geq\xi_2(t).
\]
But this last inequality is in contradiction with $\limsup_{t\rightarrow\infty}[h(t)-ct]=H^*<\infty$. We have thus proved Claim 1.

\noindent
{\bf Claim 2:}\, There exists $T_*>T$ such that
\[
h(t)-ct\geq H^*-\tau\ \mbox{for all }\ t>T_*.
\]
Since $\tau>0$ can be arbitrarily small, this claim clearly implies the validity of the lemma.

Suppose the claimed conclusion is not true. Then, in view of $\limsup_{t\rightarrow\infty}[h(t)-ct]=H^*$,
the function $[h(t)-ct]-(H^*-\tau)$ changes sign infinitely many times as $t$ increases to $\infty$.  We are going to use this fact to derive a contradiction.

Define
\[
\eta(t,x):=u(t, x+ct)-V_{l_*}(x),
\]
\[
l(t):=\min\{h(t)-ct,H^*-\tau\} \ \mbox{and} \  I(t):=[-l_*,l(t)],
\]
and let $\mathcal{Z}_{I(t)}(\eta(t,\cdot))$ denote the number of zeros of $\eta(t,\cdot)$ over $I(t)$.

Clearly
\[
\eta_t-d\eta_{xx}-c\eta_x=\Big(A(x)-b\big[u(t, x+ct)+V_{l_*}(x)\big]\Big)\eta,
\]
and for all large $t$ such that $-l_*+ct>0$, we have 
\[
\eta(t, -l_*)=u(t, -l_*+ct)>0.
\]
Moreover, $\eta(t, l(t))=0$ if and only if $[h(t)-ct]-(H^*-\tau)=0$.

Following an approach used in \cite{DLZ}, we now examine the value of $\mathcal{Z}_{I(t)}(\eta(t,\cdot))$ near any $t$ where $h(t)-ct$ crosses or touches  
the value $H^*-\tau$.
Choose $s_1>s_0\geq T$ such that $-l_*+cs_0>0$  and
\bes
\lbl{<}\;\;
-l_*<h(s)-cs<H^*-\tau\ \mbox{for} \ s\in[s_0,s_1),\; h(s_1)-cs_1=H^*-\tau.
\ees
Then 
\[
\eta(t, -l_*)=u(t, -l_*+ct)>0,\; \eta(t, l(t))=-V_{l_*}(h(t)-ct)<0 \mbox{ for } t\in [s_0, s_1).
\]
Hence we can use the zero number result (Lemma 2.2 in \cite{DLZ}) to conclude that $\mathcal{Z}_{I(s)}(\eta(s,\cdot))$ is finite and nonincreasing for $s\in(s_0,s_1)$, and
each time a degenerate zero appears in $I(s)$ for $\eta(s,\cdot)$, the value of $\mathcal{Z}_{I(s)}(\eta(s,\cdot))$ is decreased by at least $1$.
So, in the interval $(s_0,s_1)$, there can exist at most finitely many value of $s$ such that $\eta(s,\cdot)$ has a degenerate zero in $I(s)$. Thus, we can choose $\tilde s_1\in(s_0,s_1)$ such that for each $s\in[\tilde s_1,s_1)$, $\eta(s,\cdot)$ has only nondegenerate zeros in $I(s)$.
Due to
the nondegeneracy, the zeros of $\eta(s,\cdot)$, with $s\in[\tilde s_1,s_1)$, can be expressed as smooth curves:
\[
x=\gamma_1(s),\cdot\cdot\cdot,x=\gamma_m(s), \ \mbox{with} \ -l_*<\gamma_i(s)<\gamma_{i+1}(s)<l(s)\ \mbox{for}
\ i=1,2,\cdot\cdot\cdot,m-1.
\]
Similar to the proof of Lemma \ref{prop-infty}, we can show that the following limits exist:
\[
x_1=\lim_{s\rightarrow s_1^-}\gamma_1(s),\cdot\cdot\cdot,x_m=\lim_{s\rightarrow s_1^-}\gamma_m(s).
\]
If $x_i<x_{i+1}$, then $\eta(s_1,x)\neq 0$ for $x\in(x_i,x_{i+1})$, which
 follows from the strong maximum principle applied to the region
$D_i:=\{(t,x):\gamma_i(t)<x<\gamma_{i+1}(t), \tilde s_1\leq t\leq s_1\}$.
Moreover, as in the proof of Lemma \ref{prop-infty}, we can also prove
$$x_m=H^*-\tau=h(s_1)-cs_1.$$
Now, denote the zeroes of $\eta(s_1,\cdot)$ in $I(s_1)$ by
$y_1<y_2<\cdot\cdot\cdot<y_{m_1}$, where $m_1\leq m$, $y_1>-l_*$ and $y_{m_1}=H^*-\tau$.

{\it Case 1:}\, If $\eta(s_1,x)>0$ for $x\in(y_{m_1-1},y_{m_1})$, we show  that there exists $\tilde s_2>s_1$ such that $h(t)-ct>H^*-\tau$ for $t\in(s_1, \tilde s_2]$.

Indeed, in this case, for fixed $\hat x\in (y_{m_1-1}, y_{m_1})$, from $\eta(s_1, \hat x)>0$, by continuity, we can find $\tilde s_2>s_1$ close to $s_1$ such that
\[
\eta(s, \hat x)>0 \mbox{ for } s\in [s_1, \tilde s_2].
\]
We may then compare $(u(t,x), h(t))$ with $(V_{l_*}(x-ct), ct+H^*-\tau)$ by the comparison principle for free boundary problems (\cite{DL}) over the region 
\[
\Omega:=\{(x,t): ct+\hat x<x<ct+H^*-\tau,\; s_1< t\leq \tilde s_2\},
\]
 to conclude that
\[
h(t)>ct+H^*-\tau, \; u(t,x)>V_{l_*}(x-ct) \mbox{ in } \Omega.
\]
Hence  $h(t)-ct>H^*-\tau$ for $t\in(s_1,\tilde s_2]$.

{\it Case 2:}\, If $\eta(s_1,x)<0$ for $x\in(y_{m_1-1},y_{m_1})$, we can similarly show  that there exists $\tilde s_3>s_1$ such that $h(t)-ct<H^*-\tau$ for $t\in(s_1, \tilde s_3]$.

 Thus if we denote $\hat s_2=\tilde s_2$ when case 1 happens, and $\hat s_2=\tilde s_3$ when case 2 happens, we always have 
\[
\eta(t,x)\not=0 \mbox{ for } x\in\partial I(t)=\{-l_*, l(t)\},\; t\in (s_1, \hat s_2].
\]
Moreover, we can find $\epsilon_1>0$, $\epsilon_2>0$, both sufficiently small, such that
\[
\eta(t,l(t)-\epsilon_2)\not=0 \mbox{ for } t\in [s_1-\epsilon_1, \hat s_2],
\]
and
\bes\lbl{l(t)}
\eta(t,x)\not=0 \mbox{ for } t\in (s_1, \hat s_2],\; x\in [l(t)-\epsilon_2, l(t)].
\ees
Thus we can apply Lemma 2.2 of \cite{DLZ} to conclude that
\[
\mathcal{Z}_{[-l_*, l(t)-\epsilon_2]}(\eta(t,\cdot)) \mbox{ is finite and nondecreasing for } t\in (s_1-\epsilon_1, \hat s_2],
\]
and its value is decreased by at least 1 whenever $\eta(t, \cdot)$ has a degenerate zero in $[-l_*, l(t)-\epsilon_2]$.  This implies, in particular, that there can exist at most finitely many $t\in (s_1, \hat s_2]$ such that $\eta(t,\cdot)$ has a degenerate zero in $[-l_*, l(t)-\epsilon_2]$. Hence we can find $\bar s_2\in (s_1, \hat s_2]$ such that $\eta(t,\cdot)$ has only nondegenerate zeros in $[-l_*, l(t)-\epsilon_2]$ for $t\in (s_1, \bar s_2]$.
As before, for $t\in (s_1, \bar s_2]$, we can respresent the nondegenerate zeros of $\eta(t,\cdot)$ in $[-l_*, l(t)-\epsilon_2]$ by
\[
\tilde \gamma_1(t)<\tilde\gamma_2(t)<...<\tilde \gamma_p(t),
\]
and each $\tilde\gamma_i(t)$ is a smooth function for $t\in (s_1, \bar s_2]$. Moreover, $\tilde z_i:=\lim_{t\to s_1^+}\tilde \gamma_i(t)$ exists for each $i\in\{1,..., p\}$, for otherwise $w(s_1,\cdot)$ would be identically zero over some interval of  
$x$, contradicting to what we know about $w(s_1,\cdot)$. Furthermore, $\tilde z_i<\tilde z_{i+1}$ for $i\in\{1,..., p-1\}$, since otherwise we may apply the maximum principle over the region $\tilde A_i:=\{(x,t): \tilde \gamma_i(t)<x<\tilde\gamma_{i+1}(t), s_1\leq t\leq \bar s_2\}$ to deduce $w\equiv 0$ in $\tilde A_i$. Therefore $\tilde z_1<...<\tilde z_p$ are different zeros of $\eta(s_1,\cdot)$ in $[-l_*, l(s_1)-\epsilon_2]$. It follows immediately that $\{\tilde z_i:1\leq i\leq p\}\subset\{y_j: 1\leq j\leq m_1-1\}$ and hence $p\leq m_1-1$. In view of \eqref{l(t)}, we have 
\[
\mathcal{Z}_{I(t)}(\eta(t,\cdot))=\mathcal{Z}_{[-l_*, l(t)-\epsilon_2]}(\eta(t,\cdot))=p\leq m_1-1 \mbox{ for } t\in (s_1, \bar s_2].
\]
Recalling that 
\[
\mathcal{Z}_{I(t)}(\eta(t,\cdot))=m \mbox{ for } t\in (\tilde s_1, s_1),
\]
and 
\[\mbox{
$\mathcal{Z}_{I(t)}(\eta(t,\cdot))=m_1\leq m$ for $t=s_1$,}
\]
 we see that 
the value of $\mathcal{Z}_{I(t)}(\eta(t,\cdot))$ is decreased by at least 1 when $t$ increases across $s_1$, and $s_1$ is an isolated zero of the function $[h(t)-ct]-(H^*-\tau)$.

We now observe that if \eqref{<} is changed to
\[
h(s)-cs>H^*-\tau \mbox{ for } s\in [s_0, s_1),\; h(s_1)-cs_1=H^*-\tau,
\]
then the above arguments carry over and we also obtain the conclusion that
 $\mathcal{Z}_{I(t)}(\eta(t,\cdot))$ is decreased by at least 1 when $t$ increases across $s_1$, and $s_1$ is an isolated zero of the function $[h(t)-ct]-(H^*-\tau)$. The only point that requires extra attention is the following: We need to show
 $\mathcal{Z}_{I(t)}(\eta(t,\cdot))\geq 1$ for $t\in (s_0, s_1]$. Otherwise we can use the comparison principle for free boundary problems to show that $h(t)>ct+H^*-\tau$ for $t>s_1$ and $u(t,x)>V_{l_*}(x-ct)$ for $x\in [ct-l_*, ct+H^*-\tau]$ and $t>s_1$, which is in contradiction to the fact that the function $[h(t)-ct]-(H^*-\tau)$ changes sign infinitely many times as $t\to+\infty$.

By repeating the above process we can find a sequence $s_0<s_1<s_2<...<s_n<s_{n+1}<...$ such that 

(i) $
h(t)-ct=H^*-\tau \mbox{ for } t\in\{s_n: n\geq 1\},\; h(t)-ct\not=H^*-\tau \mbox{ for } t\in \cup_{n\geq 1}(s_n, s_{n+1}), 
$

 (ii) $\mathcal{Z}_{I(t)}(\eta(t,\cdot))$ is finite and nonincreasing for $t\in [s_0,\lim_{n\to\infty}s_n)$, and

(iii) $
0\leq \mathcal{Z}_{I(t)}(\eta(t,\cdot))\leq \mathcal{Z}_{I(s_1)}(\eta(s_1,\cdot))-n \mbox{ for } t\in (s_n, s_{n+1}).
$

\noindent
Here to guarantee that infinitely many such $s_n$ exist, we have used the fact that the function $[h(t)-ct]-(H^*-\tau)$
changes sign infinitely many times as $t\to\infty$,

Letting $n\to\infty$ in (iii), we deduce $\mathcal{Z}_{I(s_1)}(\eta(s_1,\cdot))=+\infty$, a contradiction. This completes the proof.
\end{proof}

The following result indicates that the situation described in Lamma \ref{prop3} actually can never happen.
\begin{lem}\lbl{H^*>L_c}\,
If $H^*:=\limsup_{t\rightarrow\infty}[h(t)-ct]<\infty$, then $H^*\leq L_{*}$.
\end{lem}
\begin{proof}\,
Suppose that the conclusion is not true, i.e., $H^*\in (L_{*}, \infty)$. Then we can apply Lemma \ref{prop3} to conclude that $\lim_{t\rightarrow\infty}[h(t)-ct]=H^*$.
Choose a sequence $\{t_n\}$ satisfying $t_n\rightarrow\infty$.
We define
\[
g(t):=h(t)-ct,\; w(t,x):=u(t,x+h(t)),\ \ t>0,\; x<0,
\]
\[
w_n(t,x):=w(t+t_n,x),\; 
g_n(t):=g(t+t_n), \ \ h_n(t):=h(t+t_n).
\]
Then 
\bes\lbl{hat-w_n}\left\{\begin{array}{lll}\medskip
\displaystyle \frac{\partial w_n}{\partial t}=d\frac{\partial^2 w_n}{\partial x^2}&\hspace{-0.3cm}+\; [c+g_n'(t)]\frac{\partial w_n}{\partial x}&\\
&+A(x+g_n(t)) w_n-b w_n^2,\ \ & t>-t_n, -h_n(t)<x<0,\\
\medskip\displaystyle  w_n(t,0)=0, \ \ & & t>-t_n,\\
\displaystyle -\mu\frac{\partial w_n}{\partial x}(t,0)&\hspace{-0.5cm}=c+g_n'(t),\ \ & t>-t_n.
\end{array}\right.
\ees
By Lemma \ref{estimate}, $\{ w_n\}$   and $\{g_n'\}$ are bounded in the $L^\infty$ norm.
By the $L^p$ estimates and Sobolev embeddings, and a standard diagonal process, there exists a subsequence of $\{w_n\}$, denoted still by $\{w_n\}$ for convenience, such that
$$ w_n\rightarrow \hat{w}\ \mbox{in}\ C_{loc}^{\frac{1+\alpha}{2},1+\alpha}(\bR^1\times(-\infty,0]),$$
where $\alpha\in(0,1)$.
By virtue of the third identity in (\ref{hat-w_n}), we have
$$g_n'(t)\rightarrow\xi(t):=-\mu\frac{\partial\hat w}{\partial x}(t,0)-c\ \mbox{in} \ C_{loc}^{\alpha/2}(\bR^1)\ \mbox{as}\ n\rightarrow\infty.$$
Since
$$g_n(t)=g_n(0)-\int_0^tg_n'(s)ds,$$
 letting $n\rightarrow\infty$, we have, in view of $\lim_{n\to\infty}g_n(t)=\lim_{s\to\infty}[h(s)-cs]=H^*$,
$$H^*=H^*-\int_0^t\xi(s)ds.$$
So for any $t\in\bR^1$, $\int_0^t\xi(s)ds=0$, which implies $\xi\equiv 0$.
Now we see that $\hat w$ satisfies
\bes\lbl{hat-w}\left\{\begin{array}{ll}\medskip
\displaystyle \frac{\partial\hat w}{\partial t}=d\frac{\partial^2\hat w}{\partial x^2}+c\frac{\partial\hat w}{\partial x}+A(x+H^*)\hat w-b\hat w^2,\ \ & t\in\bR^1, -\infty<x<0,\\
\medskip\displaystyle \hat w(t,0)=0,\ \ & t\in\bR^1,\\
\displaystyle -\mu\frac{\partial \hat w}{\partial x}(t,0)=c,\ \ & t\in\bR^1.
\end{array}\right.
\ees
Since $\hat w\geq 0$ and $\hat w_x(t,0)=-c/\mu<0$, by the strong maximum principle we must have $\hat w(t,x)>0$ for $t\in \bR^1$ and $x<0$.

Since $H^*>L_{*}$, by Lemma \ref{lem3} and Remark \ref{l}, there exists  $l_*$ such that $L(l_*)=H^*$. Fix $\tilde l>l_*$ and let $\phi\in C_0([-\tilde l-H^*,0])$ be a nonnegative function satisfying $\phi\not\equiv 0$ and $\phi(x)\leq \hat w(0,x)$ on $[-\tilde l-H^*,0]$.
Let $u_{\phi}(t,x)$ denote the unique positive solution of
\bes\lbl{u-phi}\left\{\begin{array}{ll}
\medskip\displaystyle u_t=du_{xx}+cu_x+A(x+H^*)u-bu^2,\ \ & t>0, -\tilde l-H^*<x<0,\\
\medskip\displaystyle u(t,-\tilde l-H^*)=u(t,0)=0,\ \ & t>0,\\
\displaystyle u(0,x)=\phi(x),\ \ & x\in[-\tilde l-H^*,0].
\end{array}\right.
\ees
Since $\lambda_1[-\tilde l,H^*]<\lambda_1[-l_*, H^*]=\lambda_1[-l_*, L(l_*)]<0$, by standard result for logistic equations we have
$$u_\phi\rightarrow u_*\ \mbox{in}\ C^{1+\alpha}[-\tilde l-H^*,0]\ \mbox{as}\ t\rightarrow\infty,$$
where $u_*$ is the unique positive solution of
\bess\left\{\begin{array}{ll}\medskip
\medskip\displaystyle u_*''+cu_*'+A(x+H^*)u_*-bu_*^2=0,\ \ & -\tilde l-H^*<x<0,\\
\displaystyle u_*(-\tilde l-H^*)=u_*(0)=0.
\end{array}\right.
\eess
Moreover, since $\tilde l>l_*$, by the comparison principle we find that $u_*(x)>V_{l_*}(x+H^*)$ for $x\in [-l_*-H^*, 0)$.
We may then apply the Hopf boundary lemma to deduce 
\[
u'_*(0)<V_{l_*}'(H^*)=V_{l_*}'(L(l_*))=-c/\mu.
\]
On the other hand, we may use the comparison principle  to (\ref{u-phi}) and (\ref{hat-w}) to obtain $u_\phi\leq \hat w$ for $t>0$ and $x\in [-\tilde l-H^*, 0]$, and it follows that
$$-\frac{c}{\mu}=\frac{\partial \hat w}{\partial x}(t,0)\leq \frac{\partial u_\phi}{\partial x}(t,0) \ \mbox{for} \ t>0.$$
Letting $t\rightarrow\infty$, we obtain
\[
u_*'(0)\geq -c/\mu,
\]
which contradicts
$u_*'(0)<-c/\mu$. The proof is complete.
\end{proof}

\begin{lem}\lbl{L_{*}}\,
If $\limsup_{t\rightarrow\infty}[h(t)-ct]=L_{*}$, then $\lim_{t\rightarrow\infty}[h(t)-ct]=L_{*}$.
\end{lem}

\begin{proof}\, It suffices to show that
\[
H_*:=\liminf_{t\rightarrow\infty}[h(t)-ct]=L_*.
\]
Otherwise, $H_*<L_*$.  It follows that, for any $L\in (H_*, L_*)$, the function $h(t)-ct-L$ changes sign infinitely many times as $t\to\infty$. 
Similar to the proof of Lemma \ref{prop3}, we are going to derive a contradiction from this fact.

Fix $L_0\in (H_*, L_*)$. By Lemma \ref{lem4}, there exists large $l>0$ such that  
\[
-\mu W_{l,L_0}'(L_0)<c,
\]
 where $W_{l,L}(x)$ denotes the unique positive solution of \eqref{M-L-*} with $L_*$ replaced by $L$. Clearly $W_{l,L_*}=W_l$, the unique positive solution of \eqref{M-L-*}.
By the choice of $M$ in \eqref{M-L-*} and a simple comparison argument, we easily see that $W_l(x)>V_*(x)$ for $x\in [-l, L_*)$. By the Hopf boundary lemma, we obtain $W_l'(L_*)<V_*'(L_*)=-c/\mu$. Therefore
\[
-\mu W'_{l,L_*}(L_*)=-\mu W'_l(L_*)>c.
\]
By the continuous dependence of $W_{l,L}'(L)$ on $L$, we see that there exists $L\in (L_0, L_*)$ such that 
\[
-\mu W'_{l,L}(L)=c.
\]

Let us  observe that due to the choice of $M$, we have $u(t,x)<M$ for $t>0$ and $x\in [0, h(t)]$. 
We now define
\[
\eta(t,x):=u(t, x+ct)-W_{l,L}(x),
\]
\[
l(t):=\min\{h(t)-ct,L\}\ \mbox{and}\ I(t)=[-l,l(t)].
\]
Then for large $t$, say $t\geq T_0$, we have $x+ct>0$ for $x\in I(t)$ and hence
\[
\eta(t, -l)=u(t, -l+ct)-M<0.
\]
Moreover, $\eta(t, l(t))=0$ if and only if $h(t)-ct-L=0$. Since $h(t)-ct-L$ changes sign infinitely many times as $t\to+\infty$, we 
may repeat the arguement in the proof of Claim 2 of Lemma
 \ref{prop3} to derive a contradiction. The details are omitted.
\end{proof}

\subsection{The case of spreading}

We start by giving some sufficient conditions for 
\[
\lim_{t\to\infty}[h(t)-ct]=+\infty.
\]
  Let $V_0$ be the positive solution of
\eqref{V-l-L} with $l=0$ and $L=L(0)$, so that we have $-\mu V_0'(L(0))=c$.

\begin{lem}\lbl{spreading-2}
Suppose there exists $t_0\geq 0$ such that
\bes\lbl{2-condition}h(t_0)-ct_0\geq L(0)\ \ \mbox{and}\ \ u(t_0,x)\geq V_0(x-ct_0)\ \ \mbox{for}\ x\in[ct_0,ct_0+L(0)].\ees
Then $\lim_{t\to\infty}[h(t)-ct]=+\infty$.
\end{lem}
\begin{proof}
Define
$$
\xi_1(t):=c(t+t_0),\; \ \xi_2(t):=c(t+t_0)+L(0)\ \mbox{for}\ t>0,
$$
and
$$
\underline u(t,x):=V_0(x-c(t_0+t))\ \mbox{for}\ t>0\ \mbox{and}\ x\in[\xi_1(t), \xi_2(t)].
$$
Clearly $\underline u$ satisfies 
\[ 
u_t-du_{xx}=au-bu^2 \mbox{ for } t>0,\; x\in [\xi_1(t), \xi_2(t)],
\]
and
\[
\underline u(t, \xi_1(t))=\underline u(t, \xi_2(t))=0,
 \]
\[ 
\xi_2'(t)=c=-\mu V_0'(L(0))=-\mu \underline u_x(t, \xi_2(t)).
\]
Hence in view of (\ref{2-condition}), and $u(t, \xi_1(t))>0$,  and $A(x-c(t_0+t))=a$ for $x\in [\xi_1(t),\xi_2(t)]$ and $t>0$, we can use the comparison  principle for free boundary problems to obtain
$$h(t+t_0)\geq \xi_2(t)\ \mbox{for}\ t>0$$
and
$$u(t+t_0,x)\geq \underline u(t,x)\ \mbox{for}\ t>0\ \mbox{and}\ x\in(\xi_1(t), \xi_2(t)).$$
Since $u(t, \xi_1(t))>0=\underline u(t, \xi_1(t))$, by the strong maximum principle we further obtain
\[
u(t+t_0,x)> \underline u(t,x)=V_0(x-c(t_0+t))\ \mbox{for}\ t>0,\;  x\in(\xi_1(t), \xi_2(t)),
\]
and 
\[
h(t+t_0)>\xi_2(t)=c(t+t_0)+L(0) \mbox{ for $t>0$.}
\]

To stress the dependence of $L(0)$ and $V_0$ on $c$, we now rewrite them as $L(0)=L_c(0)$ and $V_0=V_{0,c}$.
Then
by the  Hopf lemma and continuity of $(L_c(0), V_{0,c})$ on $c$, for fixed $\tilde t_0>t_0$, we can find 
 $\tilde c>c$ but very close to $c$,
such that
\[
h(\tilde t_0)\geq \tilde c\, \tilde t_0+L_{\tilde c}(0)
\]
and
\[
u(\tilde t_0,x)\geq V_{0,\tilde c}(x-\tilde c\, \tilde t_0)\ \ \mbox{ for }\ \tilde c\,\tilde t_0<x<\tilde c\, \tilde t_0+L_{\tilde c}(0).
\]

We may now repeat the above comparison argument, but with $(t_0, c)$ replaced by $(\tilde t_0, \tilde c)$, to deduce that
$$h(t+\tilde t_0)\geq \tilde c\,(t+\tilde t_0)+L_{\tilde c}(0)\ \ \mbox{for}\ t>0,$$
and 
$$u(t+\tilde t_0,x)\geq V_{0, \tilde c}(x-\tilde c(t+\tilde t_0))$$
for $t>0$, $\tilde c (t+\tilde t_0)<x<\tilde c (t+\tilde t_0)+L_{\tilde c}(0)$.
We thus obtain
$\lim_{t\rightarrow\infty}[h(t)-ct]=\infty$. 
\end{proof}

\begin{remark}\lbl{suff-cond}{\rm
Let us note that if we take $t_0=0$ in \eqref{2-condition}, then this sufficient condition is reduced to a condition on the initial values of \eqref{MP}:
\bes\lbl{0-condition}
h(0)\geq L(0) \mbox{ and } u(0,x)\geq V_0(x) \mbox{ in } [0, L(0)].
\ees
}
\end{remark}

\begin{theo}\lbl{1-2-3}
If $\limsup_{t\rightarrow\infty}[h(t)-ct]>L_*$, then
$\lim_{t\to\infty}h(t)/t=c_0$ and for any given small $\epsilon>0$,
\bes
\lbl{0-ct}
 \lim_{t\rightarrow+\infty}\left[\sup_{0\leq x\leq(c-\varepsilon)t} u(t,x)\right]=0,
\ees
and
\bes\lbl{ct-h(t)}
  \lim_{t\rightarrow+\infty}\left[\sup_{(c+\varepsilon)t\leq x\leq (1-\varepsilon)h(t)}\left|u(t,x)-\frac{a}{b}\right|\right]=0.
  \ees
\end{theo}

We will prove \eqref{0-ct} and \eqref{ct-h(t)} separately. In fact, for later applications, 
we will prove in Lemma \ref{0-ct-M} below a slightly stronger version of \eqref{0-ct} without using the assumption $\limsup_{t\rightarrow\infty}[h(t)-ct]=\infty$.
Moreover, different from the rest of this section,  the assumption $0<c<c_0$  is also not required in Lemma \ref{0-ct-M}.

\begin{lem}
\lbl{0-ct-M}
Suppose $c>0$ and $(u,h)$ is the unique solution of \eqref{MP}. For $M>0$ define
\[
\epsilon(M):=\limsup_{t\to\infty}\left[\sup_{0\leq x\leq ct-M}u(t,x)\right].
\]
Then  $\lim_{M\to\infty}\epsilon(M)=0$.
Here we understand that $u(t,x)=0$ for $x\geq h(t)$.
\end{lem}
\begin{proof}\,
Let us define
\[
\tilde{u}(t,x)=
\begin{cases}
u(t,x),\ \ &t>0,\  x\geq 0,\\
u(t,-x),\ \ &t>0,\  x<0.
\end{cases}
\]Then $\tilde u$ satisfies 
\[
\left\{
\begin{array}{ll}
\tilde u_t-d\tilde u_{xx}=A(|x|-ct)\tilde u-b \tilde u^2, & t>0,\; x\in (-h(t), h(t)),\vspace{0.2cm}\\
\tilde u(t, -h(t))=\tilde u(t, h(t))=0, & t>0.
\end{array}
\right.
\]
Suppose that  $\lim_{M\to\infty}\epsilon(M)=0$ does not hold. Then  there exist $\sigma_0>0$ and a sequence $M_n\to\infty$
such that $\epsilon(M_n)> \sigma_0$ for all $n\geq 1$. Therefore we can find a
 sequence of points $(t_n,x_n)$ satisfying $t_n\rightarrow \infty$ and $ x_n\in[0, ct_n-M_n]$ such that
\[
\tilde{u}(t_n,x_n)> \sigma_0.
\]
Set
\[
 \tau_n:=t_n-\frac{M_n}{3c},
\]
and  for  $ 0<t<t_n-\tau_n$ and $x\in \left(-\frac{M_n}{2}, \frac{ M_n}{2}\right)\cap \big(-h(t+\tau_n)-x_n, h(t+\tau_n)-x_n\big)$,
define
\[
v_n(t,x):=\tilde{u}(t+\tau_n,x+x_n).
\]
 Obviously, 
\[
  t_n-\tau_n=\frac{M_n}{3c}\to\infty, \;  \tau_n\geq \frac{2}{3}t_n \to\infty,
\]
and for
 $0\leq t\leq t_n-\tau_n$ and $-\frac{ M_n}{2} <x<\frac{ M_n}{2}$, we have
\[
|x+x_n| -c(t+\tau_n)\leq \frac{M_n}{2}+ct_n-M_n-c\tau_n= -\frac{M_n}{6}.
\]
Hence for all large $n$, $A(|x+x_n| -c(t+\tau_n))=a_0$, and
$v_n(t,x)$ satisfies,  for  $ 0<t<t_n-\tau_n$ and $x\in \left(-\frac{M_n}{2}, \frac{M_n}{2}\right)\cap \big(-h(t+\tau_n)-x_n, h(t+\tau_n)-x_n\big)$,
\[
\frac{\partial v_n}{\partial t}-d\frac{\partial^2v_n}{\partial x^2}=a_0 v_n-bv_n^2,
\]
and 
\[
v_n(t_n-\tau_n,0)=\tilde{u}(t_n,x_n)> \sigma_0.
\]

Set $M_0:=\max\{a/b, \|u_0\|_\infty\}$ and let  $\tilde v_{l}$ be the unique positive solution of the following initial-boundary value problem
\begin{equation}\label{eq4.6}
\begin{cases}
v_t-dv_{xx}=-bv^2,\ \ \ \ &t>0,\,-l<x<l,\\
v(t,-l)=v(t,l)=M_0,\ \ & t>0,\\
v(0,x)=M_0, \ \ & -l<x<l.
\end{cases}
\end{equation}
Then, $\tilde v_l(t,\cdot)$ converges to $v^*_l$ uniformly in $[-l,l]$ as $t\to\infty$, where $v_l^*$ is the unique positive solution of
\[
\begin{cases}
 dv_{xx}=bv^2,\ \ & -l<x<l,\\
 v(-l)=v(l)=M_0.
\end{cases}
\]
Using Lemma 2.2 in \cite{DM}, we have $\lim_{l\rightarrow+\infty}v_l^*(x)=0$ uniformly in any compact subset of $\bR^1$.
Fix large $l>0$ such that $v^*_l(0)<\sigma_0/2$. Then, we choose large $t_0>0$ such that for $t\geq t_0$, $\tilde v_l(t,0)<\sigma_0$. A simple comparison consideration yields $\tilde u\leq M_0$. Hence for all large $n$ such that $\tilde M_n/2>l$, we can compare $v_n$ with $\tilde v_l$ by the comparison principle to obtain
\[
v_n(t,x)\leq \tilde v_{l}(t,x)\ \ \mbox{for}\ 0<t<t_n-\tau_n,\ x\in [-l, l].
\]
Since $t_n-\tau_n\to\infty$ as $n\to\infty$, we have $t_n-\tau_n>t_0$ for all large $n$, and hence
\[
\tilde u(t_n,x_n)=v_n(t_n-\tau_n,0)\leq \tilde v_l(t_n-\tau_n,0)<\sigma_0 \mbox{ for all large } n,
\]
which contradicts $\tilde u(t_n,x_n)> \sigma_0$.
This completes the proof.
 \end{proof}

\begin{proof}[{\bf Proof of Theorem \ref{1-2-3}}]
By Lemma \ref{H^*>L_c}, we necessarily have $\limsup_{t\to\infty}[h(t)-ct]=+\infty$. Hence we can apply
 Lemma \ref{speed} to conclude that $\lim_{t\rightarrow\infty}\frac{h(t)}{t}=c_0$, and
 \eqref{0-ct} is clearly a consequence of Lemma \ref{0-ct-M}. 

It remains to prove \eqref{ct-h(t)}, and
this will be accomplished by an indirect argument. Suppose that there exists $\sigma_0>0$ such that for some small $\epsilon_0>0$ and
some sequence of points $(t_n,x_n)$ satisfying $t_n\rightarrow\infty,\,x_n\in[(c+\epsilon_0)t_n,(1-\epsilon_0)h(t_n)]$, we have
\[
|u(t_n,x_n)-\frac{a}{b}|>\sigma_0.
\]

Choose $\delta>0$ small so that
\[
1-\frac{\epsilon_0}{3}<(1-\delta)\Big(1-\frac{\epsilon_0}{4}\Big),\; c+\frac{\epsilon_0}{2}>(1-\delta)\Big(c+\frac{\epsilon_0}{4}\Big).
\]
Then define
\[
\gamma_n:=(1-\delta)t_n,
\]
\[
\Omega_n:=\big\{(t,x): 0\leq t\leq t_n-\gamma_n, \;-\frac{\epsilon_0}{2}t_n<x<\frac{\epsilon_0}{2}h(t_n)\big\},
\]
and
\[
v_n(t,x):=u(t+\gamma_n,x+x_n)\ \ \mbox{for}\ (t,x)\in\Omega_n.
\]
Then we have
\bes \lbl{a-b-sigma}
|v_n(t_n-\gamma_n,0)-\frac{a}{b}|>\sigma_0 \ \mbox{for all}\ n.
\ees
Moreover, for $(t,x)\in\Omega_n$,
\[
x+x_n-c(t+\gamma_n)\geq -\frac{\epsilon_0}{2}t_n+(c+\epsilon_0)t_n-ct_n=\frac{\epsilon_0}{2}t_n\to\infty,
\]
and in view of our choice of $\delta$ and the property $\lim_{t\to\infty}h(t)/t=c_0$, we also have, for all large $n$,
\bes\lbl{x+x_n<}
x+x_n\leq \left(1-\frac{\epsilon_0}{2}\right)h(t_n)\leq \left(1-\frac{\epsilon_0}{3}\right)c_0t_n\leq \left(1-
\frac{\epsilon_0}{4}\right)c_0\gamma_n,
\ees
\bes
\lbl{x+x_n>}
x+x_n\geq \left(c+\frac{\epsilon_0}{2}\right)t_n>\left(c+\frac{\epsilon_0}{4}\right)\gamma_n.
\ees
Therefore, for $(t,x)\in\Omega_n$ with large $n$, 
\[
A(x+x_n-c(t+\gamma_n))=a
\]
and $v_n$ satisfies
\[
\frac{\partial v_n}{\partial t}=d\frac{\partial^2v_n}{\partial x^2}+av_n-bv_n^2.
\]

We recall that the comparison principle gives $\tilde u(t,x)\leq \bar{w}(t)$ for $t>0$ and $x\in [0,h(t)]$, where
$
  \bar{w}(t)$
is the solution of the problem
\[
\frac{d\bar{w}}{dt}=a\bar{w}-b\bar{w}^2,\ t>0;\ \ \bar{w}( 0)=M.
\]
Since $\lim_{t\rightarrow\infty} \bar{w}(t) =\frac{a}{b}$, we deduce
\begin{equation}\label{eq4.8}
 \overline{\lim}_{t\rightarrow\infty} u(t,x)\leq \frac{a}{b}\ \ \  \text{uniformly for}\ \  x\in [0,h(t)].
\end{equation}
In view of (\ref{a-b-sigma}), this implies 
\bes\lbl{v-n-sigma}
v_n(t_n-\gamma_n,0)<\frac ab-\sigma_0\  \mbox{for all large } \ n.
\ees

Now, choose large $l$ satisfying $\tilde w_l^*(0)>a/b-{\sigma_0}/{2}$ (see \cite{D}), where $\tilde w_l^*$ is the unique positive solution of
\[
\begin{cases}
 -dw''=aw-bw^2,\ \ & -l<x<l,\\
 w(-l)=w(l)=0.
\end{cases}
\]
Fix such an $l$. For all large $n$ we have $\frac{\epsilon_0}{2}t_n>l$ and
\begin{equation}\lbl{v-n-equation-2}
\begin{cases}
\frac{\partial v_n}{\partial t}=d\frac{\partial^2v_n}{\partial x^2}+av_n-bv_n^2,\ \ \ &0<t<t_n-\gamma_n,\ -l<x<l,\\
v_n(t,-l)>0,\ v_n(t,l)>0,\ \ \ & 0<t<t_n-\gamma_n,\\
v_n(0,x)=u(\gamma_n,x+x_n),\ \ & -l<x<l.
\end{cases}\end{equation}

We will show that, there exists $\beta>0$ such that for all large $n$, say $n\geq n_0$, we have
\bes\lbl{beta}
u(\gamma_n,x+x_n)\geq  \beta \mbox{ for } -l<x<l.
\ees

Assuming \eqref{beta}, we now derive a contradiction.
Let $\tilde w_l(t,x)$ be the unique positive solution of
\begin{equation}\lbl{tilde-w-equation}
\begin{cases}
w_t=dw_{xx}+aw-bw^2,\ \ \ &t>0,\ -l<x<l,\\
w(t,-l)=0,\ w(t,l)=0,\ \ \ & t>0,\\
w(0,x)=\beta,\ \ & -l<x<l.
\end{cases}
\end{equation}
Since $\tilde w_l(t,\cdot)$ converges to $\tilde w_l^*$ uniformly as $t\to\infty$, there exists $T_*>0$ such that
\bes\lbl{t>T*}
\tilde w_l(t,0)>\frac{a}{b}-{\sigma_0}\ \mbox{when}\ t\geq T_*.
\ees
Due to \eqref{beta}, we can apply the comparison principle to (\ref{v-n-equation-2}) and (\ref{tilde-w-equation}) to conclude that
\[
\tilde w_l(t,x)\leq v_n(t,x)\ \mbox{for}\ \ 0\leq t\leq t_n-\gamma_n, \ -l<x<l.
\]
In view of $t_n-\gamma_n\to\infty$ and (\ref{t>T*}), we thus obtain
\[
v_n(t_n-\gamma_n,0)\geq \tilde w_l(t_n-\gamma_n, 0)>\frac{a}{b}-{\sigma_0} \mbox{ for all large } n,
\]
which contradicts (\ref{v-n-sigma}).

To complete the proof of the theorem, we still have to show \eqref{beta}. This will be done by a careful examination of the 
  proof of Lemma \ref{speed}. We first observe that for all large $n$, 
\[
[-l, l]\subset \left[-\frac{\epsilon_0}{2}t_n, \frac{\epsilon_0}{2}h(t_n)\right],
\]
and hence by \eqref{x+x_n<} and \eqref{x+x_n>}, we have
\bes\lbl{x+x_n}
\left(c+\frac{\epsilon_0}{4}\right)\gamma_n\leq x+x_n\leq \left(1-\frac{\epsilon_0}{4}\right)c_0\gamma_n \mbox{ for } x\in [-l, l]
\mbox{ and all large } n.
\ees
To prove \eqref{beta}, we need to show that the estimate  for $u$ obtained by the comparison argument in the proof of Lemma
\ref{speed} can be made uniform in $\tilde c\in I_0:= [c+\frac{\epsilon_0}{6}, \left(1-\frac{\epsilon_0}{6}\right)c_0]$.

We now start the examination of the proof of Lemma \ref{speed}. Firstly it is easily seen that the $\epsilon>0$ in \eqref{ab-epsilon} can be chosen independent of $\tilde c\in I_0$. Next since $\tilde L=L_{\tilde c}(0)$ has a common upper bound for $\tilde c\in I_0$, we can choose $\tilde l$ there independent of $\tilde c$.
 It follows that
 the number $T_1$ in the proof there, and hence $\psi_0$ and $T=T(\psi_0)$ are independent of $\tilde c$. 
Therefore $T_0=T+T_1$ is independent of $\tilde c$. We now obtain from \eqref{u>} that, for every $\tilde c\in I_0$,
\[
u(t+T_0, x)\geq \tilde V(x-\tilde c t-\tilde l-cT_0+\frac{\tilde L}{2})
\]
for $t\geq 0$ and $x\in [\tilde c t+\tilde l+cT_0-\frac{\tilde L}{2}, \tilde c t+\tilde l+cT_0+\frac{\tilde L}{2}]$.
Taking $x=\tilde c t+\tilde l+cT_0$ we obtain
\[
u(t+T_0, \tilde c t+\tilde l+cT_0)\geq \tilde V\Big(\frac{\tilde L}{2}\Big) \mbox{ for } t>0,\; \tilde c\in I_0.
\]
Using the earlier notation $\tilde V=V_{0,\tilde c}$ and $\tilde L=L_{\tilde c}(0)$, we find that
there exists $\beta>0$ such that
\[
\tilde V\Big(\frac{\tilde L}{2}\Big)=V_{0,\tilde c}\left(\frac{L_{\tilde c}(0)}{2}\right)\geq \beta \mbox{ for } \tilde c\in I_0.
\]
Hence, if we take $t+T_0=\gamma_n$, then 
\[
\tilde c t+\tilde l+cT_0=\tilde c \gamma_n+(c-\tilde c)T_0+\tilde l,
\]
and 
\bes\lbl{u-gamma_n}
u(\gamma_n, \tilde c \gamma_n+(c-\tilde c)T_0+\tilde l)\geq \beta \mbox{ for all large } n \mbox{ and all } \tilde c\in I_0.
\ees
Due to our choice of $I_0$, for all large $n$, 
\[
\Big\{\tilde c \gamma_n+(c-\tilde c)T_0+\tilde l: \tilde c\in I_0\Big\}\supset \left[\left(c+\frac{\epsilon_0}{4}\right)\gamma_n, \left(1-\frac{\epsilon_0}{4}\right)c_0\gamma_n\right].
\]
Thus, in view of \eqref{x+x_n},  the required estimate \eqref{beta} follows from \eqref{u-gamma_n}. The proof is complete.
\end{proof}

\subsection{The case of vanishing}

We first give a result which does not require $c<c_0$.
\begin{lem}\lbl{vanishing-h} For any $c>0$, the unique solution $(u,h)$ has the following property:

$\lim_{t\rightarrow\infty}\left[\max_{0\leq x\leq h(t)}u(t,x)\right]=0$ if and only if $h_\infty:=\lim_{t\to\infty}h(t)<\infty$.
\end{lem}
\begin{proof}\,
 Suppose $\lim_{t\rightarrow\infty}\left[\max_{0\leq x\leq h(t)}u(t,x)\right]=0$. We first prove that
\bes\lbl{h'}
\lim_{t\rightarrow\infty}h'(t)=0.
\ees
Since $h'(t)\geq 0$,  it suffices to show $\limsup_{t\rightarrow\infty}h'(t)\leq 0$.
If this is not true, then there exist $\epsilon_0>0$ and  a sequence $\{t_n\}$ such that
\[
\lim_{n\to\infty}t_n=\infty\ \ \mbox{and}\ h'(t_n)\geq\epsilon_0 \ \mbox{for all} \ n.
\]
By (\ref{MP}), we have $u_x(t_n,h(t_n))=-h'(t_n)/\mu\leq-\epsilon_0/\mu$.
Let $w(t,y)=u(t,y+h(t))$ for $t>0$ and $-h(t)\leq y\leq 0$.
Then $w(t,y)$ satisfies $w_y(t_n,0)\leq-\epsilon_0/\mu$ and
\bess\left\{\begin{array}{ll}
\medskip\displaystyle w_t=dw_{yy}+h'(t)w_y+A(y+h(t)-ct)w-bw^2,\ \ & t>0, -h(t)<y<0,\\
\medskip\displaystyle w_y(t,-h(t))=w(t,0)=0,\ \ & t>0,\\
\displaystyle w(0,y)=u_0(0,y+h_0),\ \ & -h_0\leq y\leq 0.
\end{array}\right.
\eess
Since $h'(t)$, $A(y+h(t)-ct)$ and $w(t,x)$ are all bounded in the $L^\infty$ norm, for any fixed $L\in(0,h_0]$, by $L^p$ estimates and Sobolev embeddings, there exist positive constants $\alpha\in (0,1)$ and $D>0$ such that
\bes\lbl{w-leq-D}
\|w\|_{C^{\frac{1+\alpha}{2},1+\alpha}{([1,\infty)\times[-L,0]})}\leq D.
\ees
Since $\lim_{t\rightarrow\infty}\left[\max_{0\leq x\leq h(t)}u(t,x)\right]=0$, necessarily $w(t_n,x)$ converges to $0$ uniformly in $[-L,0]$ as $n\rightarrow\infty$. By  (\ref{w-leq-D}) and a standard compactness consideration, there exists a subsequence of $\{t_n\}$, still denoted by itself, such that
 $$w(t_n,y)\rightarrow 0\ \mbox{in}\ C^1([-L,0])\ \mbox{as}\ n\rightarrow\infty.$$
It follows that $w_y(t_n,0)\rightarrow 0$, which is a contradiction to $w_y(t_n,0)\leq-\epsilon_0/\mu$. This proves \eqref{h'}.

We next show that $h_\infty<+\infty$.
Take
 \[
m:=\frac{d\pi}{9\mu}.
\]
 Due to $\lim_{t\rightarrow\infty}\left[\max_{0\leq x\leq h(t)}u(t,x)\right]=0$ and (\ref{h'}), there exists $T>0$ such that
\[
\mbox{ $u(T,x)\leq \frac{m}{\sqrt{2}}$ for $x\in[0,h(T)]$ and $h(t)<ct-l_0$ for all $t\geq T$.}
\]
Therefore $A(x-ct)=a_0$ for $t\geq T$ and $x\in [0, h(t)]$, and so 
\bess\left\{\begin{array}{ll}
\medskip\displaystyle u_t=du_{xx}+a_0 u-bu^2,\ \ & t>T, 0<x<h(t),\\
\medskip\displaystyle u_x(t,0)=u(t,h(t))=0, \ \ & t>T, \\
\displaystyle-\mu u_x(t,h(t))=h'(t) \ \ & t>T.
\end{array}\right.
\eess
Set 
$$
\alpha:=\frac{d\pi^2}{36h(T)^2},\; \tilde h(t):=h(T)(3-e^{-\alpha t}),
$$
$$
\tilde u(t,x):=me^{-\alpha t}\cos\left(\frac{\pi}{2}\frac{x}{\tilde h(t)}\right).
$$
A direct calculation gives
$$
\tilde u(0,x)\geq m\cos \left(\frac{\pi}{4}\right)= \frac{m}{\sqrt{2}}\geq u(T,x) \mbox{ for } \ 0\leq x\leq h(T),
$$
$$
\tilde u_x(t,0)=0,\ \ \tilde u(t,\tilde h(t))=0\ \mbox{for} \ t>0,
$$
$$
-\mu\tilde u_x(t,\tilde h(t))=\mu me^{-\alpha t}\frac{\pi}{2\tilde h(t)}\leq\mu m\frac{\pi}{4h(T)}e^{-\alpha t}
= \alpha h(T)e^{-\alpha t}=\tilde h'(t)\ \mbox{for}\ t>0,
$$
and
\bess
&&\tilde u_t-d\tilde u_{xx}-a_0 \tilde u+b\tilde u^2\\
&=&\left(-\alpha+\frac{\pi x\tilde h'(t)}{2\tilde h(t)^2}\tan(\frac{\pi x}{2\tilde h(t)})
+\frac{d\pi^2}{4\tilde h(t)^2}-a_0+bme^{-\alpha t}\cos(\frac{\pi x}{2\tilde h(t)})\right)\tilde u\\
&\geq&\left(\frac{d\pi^2}{36h(T)^2}-\alpha\right)\tilde u=0 \ \mbox{for} \ t>0, 0<x<\tilde h(t).
\eess
So, by the comparison principle,
\[
h(t+T)\leq \tilde h(t)=h(T)(3-e^{-\alpha t})\leq 3h(T)\ \mbox{for}\ t>0,
\]
\[
u(t+T,x)\leq \tilde u(t,x)\ \mbox{for}\ t>0 \ \mbox{and}\ 0\leq x\leq h(t+T).
\]
The first inequality clearly implies $h_\infty<\infty$.

Conversely, suppose $h_\infty<\infty$.  Then from Lemma \ref{0-ct-M} we immediately
obtain 
\[
\lim_{t\rightarrow\infty}\left[\max_{0\leq x\leq h(t)}u(t,x)\right]=0.
\]
The proof is complete.
\end{proof}

\begin{theo}\lbl{H^*<L_{c,*}}
If $\limsup_{t\rightarrow\infty}[h(t)-ct]<L_{*}$, then $h_\infty:=\lim_{t\to\infty}h(t)<+\infty$ and 
\bes\lbl{u-0}
\lim_{t\rightarrow\infty}\left[\max_{0\leq x\leq h(t)}u(t,x)\right]=0.
\ees
\end{theo}
\begin{proof}\,
By Lemma \ref{vanishing-h}, it suffices to prove \eqref{u-0}.
Denote  $H^*:=\limsup_{t\rightarrow\infty}[h(t)-ct]$. If $H^*=-\infty$, then \eqref{u-0} follows immediately from Lemma \ref{0-ct-M}.

Suppose next $H^*>-\infty$. Fix a constant $L\in(H^*,L_{*})$.
By Lemma \ref{lem4}, there exists
 $l>-H^*$ large such that $-\mu W_{l,L}'(L)<c$, where $W_{l,L}$ is given in Lemma \ref{lem4}.
Let $M_1>M$, $L_1>L$, and denote by $\mathcal{V}_c$  the unique positive solution of
\bes\lbl{M-1-L-1}\left\{\begin{array}{ll}
\medskip\displaystyle dV''+cV'+A(x)V-bV^2=0 \mbox{ for } t>t_0, -l<x<L_1,\\
\displaystyle V(-l)=M_1,\ V(L_1)=0.
\end{array}\right.
\ees
By continuity, $-\mu\mathcal{V}'_c(L_1)<c$ provided that
 $M_1$ is close enough to $M$ and $L_1$ is close enough to $L$. We now fix $M_1>M$ and $L_1>L$ so that
$-\mu\mathcal{V}'_c(L_1)<c$ holds.

By the comparison principle, there is small $\epsilon_0>0$ such that
\bess
\mathcal{V}_c(x)>W_{l,L}(x)+2\epsilon_0\ \ \mbox{for}\ x\in[-l,L].
\eess
By continuity, there is small $\delta>0$ such that
\bes\lbl{V-c-delta}
\mathcal{V}_{c-\delta}(x)>W_{l,L}(x)+\epsilon_0\ \ \mbox{for}\ x\in[-l,L]
\ees
and
\[
-\mu\mathcal{V}_{c-\delta}'(L_1)<c-\delta,
\]
where $\mathcal{V}_{c-\delta}$ is the unique positive solution of (\ref{M-1-L-1}) with $c$ replaced by $c-\delta$.

Since $\limsup_{t\rightarrow\infty}[h(t)-ct]=H^*<L$, there is $t_0>0$ so that
$ct_0>l$ and $h(t)-ct<L$ for all $t\geq t_0$.
We now consider the auxiliary problem
\bes\lbl{v(t,x;c)}\left\{\begin{array}{ll}
\medskip\displaystyle v_t=dv_{xx}+cv_x+A(x)v-bv^2,\ \ t>t_0, -l<x<L,\\
\medskip\displaystyle v(-l)=M, \ \ v(L)=0,\\
\displaystyle v(t_0,x)=M,
\end{array}\right.
\ees
which has a unique positive solution $v(t,x;c)$, and by the comparison principle,
\[
v(t,x;c)>u(t,x+ct)\ \mbox{for}\ t>t_0, -l<x<h(t)-ct.
\]
Clearly,
\bes\lbl{W_cl}
v(t,x;c)\rightarrow W_{l,L}(x)\ \mbox{in}\ C^{2}([-l,L])\ \mbox{as}\ t\rightarrow\infty.
\ees
Therefore, there is $t_1>t_0$ such that
\[
W_{l,L}(x)+\epsilon_0\geq u(t_1,x+ct_1)\ \ \mbox{for}\ -l\leq x\leq h(t_1)-ct_1.
\]
From (\ref{V-c-delta}), it follows that
\[
\mathcal{V}_{c-\delta}(x)\geq u(t_1,x+ct_1)\ \ \mbox{for}\ -l\leq x\leq h(t_1)-ct_1.
\]

We now define
\[
\tilde u(t,x):=\mathcal{V}_{c-\delta}(x-ct_1-(c-\delta)t)
\]
\[\xi_1(t):=ct_1+(c-\delta)t-l, \ \xi_2(t):=ct_1+(c-\delta)t+L_1.
\]
Then,
\bess
\tilde u_t&=&d\tilde u_{xx}+A(x-ct_1-(c-\delta)t)u-bu^2\\
&\geq& d\tilde u_{xx}+A(x-c(t+t_1))u-bu^2\ \ \mbox{for}\ t>0, x\in(\xi_1(t),\xi_2(t);
\eess
\[
\tilde u(t,\xi_1(t))=M_1>u(t+t_1,\xi_1(t)),\ \ t>0;
\]
\[
\tilde u(t,\xi_2(t))=\mathcal{V}_{c-\delta}(L_1)=0,\ \ t>0;
\]
\[
-\mu\tilde u_x(t,\xi_2(t))=-\mu \mathcal{V}'_{c-\delta}(L_1)<c-\delta=\xi_2'(t),\ \ t>0;
\]
\[
\xi_2(0)=ct_1+L>h(t_1);
\]
\[
\tilde u(0,x)\geq u(t_1,x) \ \mbox{for}\ x\in(\xi_1(0),\xi_2(0)).
\]
By the comparison principle, we obtain
\[
h(t+t_1)\leq \xi_2(t)=ct_1+(c-\delta)t\ \mbox{for}\ t>0,
\]
\[
u(t+t_1,x)\leq \tilde u(t,x) \ \mbox{for}\ t>0,\ \xi_1(t)<x<h(t).
\]
But the first inequality implies $H^*=-\infty$. This contradiction implies that the case $-\infty<H^*<L_*$ cannot happen.
The proof is complete.
\end{proof}

\subsection{The case of borderline spreading}

\begin{theo}\lbl{Transition}\,
If $\limsup_{t\rightarrow\infty}[h(t)-ct]=L_{*}$, then $\lim_{t\rightarrow\infty}[h(t)-ct]=L_{*}$ and
\bes\lbl{0-h}
\lim_{t\rightarrow\infty} \left[\max_{0\leq x\leq h(t)}\big|u(t,x)-V_{*}(x-h(t)+L_*)\big|\right]=0.
\ees
\end{theo}
\begin{proof}\, The first conclusion has been proved in Lemma \ref{L_{*}}. It remains to prove \eqref{0-h}.

 For $ t>0$ and $ -h(t)<x<0$, define
\[ 
g(t):=h(t)-ct,\; w(t,x):=u(t,x+h(t)).
\]
Let $\{t_n\}$ be an arbitrary sequence  satisfying $t_n\rightarrow\infty$,
and define
\[
g_n(t):=g(t+t_n),\;\; w_n(t,x):=w(t+t_n,x).
\]
Then, $w_n$ and $g_n$ satisfy (\ref{hat-w_n}). By the arguments in the proof of Lemma \ref{H^*>L_c}, there exists a subsequence of $\{w_n\}$, denoted still by $\{w_n\}$ for convenience, and $\alpha\in (0,1)$, such that $w_n\rightarrow\hat w$ in $C^{\frac{1+\alpha}{2},1+\alpha}_{loc}(\bR^1\times(-\infty,0])$, $g_n'(t)\rightarrow 0$ in $C_{loc}^{\alpha/2}(\bR)$, and
$\hat w$ satisfies
\bess
\left\{\begin{array}{ll}\medskip
\displaystyle \frac{\partial\hat w}{\partial t}=d\frac{\partial^2\hat w}{\partial x^2}+c\frac{\partial\hat w}{\partial x}
                             +A(x+L_{*})\hat w-b\hat w^2,\ \ & t\in\bR^1, -\infty<x<0,\\
\medskip\displaystyle \hat w(t,0)=0,\ \ & t\in\bR^1,\\
\displaystyle -\mu\frac{\partial \hat w}{\partial x}(t,0)=c,\ \ & t\in\bR^1.
\end{array}\right.
\eess

We claim that 
\bes\lbl{hat-w-2}
\hat w(t,x)\equiv V_*(x+L_*).
\ees
We will use an argument from \cite{GLZ} to prove this claim. Arguing indirectly, we assume that \eqref{hat-w-2} is not true.
Then there exist $t_0\in\bR^1$ and $x_0<0$ such that $\hat w(t_0,x_0)\not=V_*(x_0+L_*)$. 
By continuity, there exists $\epsilon_0>0$ such that
\[
\hat w(t, x_0)\not=V_*(x_0+L_*) \mbox{ for } t\in [t_0-\epsilon_0, t_0+\epsilon_0].
\]
We now consider the function
\[
\eta(t,x):=\hat w(t,x)-V_*(x+L_*) \mbox{ for } (t,x)\in \Omega_0:=[t_0-\epsilon_0, t_0+\epsilon_0]\times [x_0,0].
\]
Clearly
\[
\eta_t=d\eta_{xx}+c\eta_x+\Big(A(x+L_*)-b\big[\hat w+V_*(\cdot+L_*)\big]\Big)\eta\;\; \mbox{ in } \Omega_0,
\]
\[
\hspace{-2cm}\eta(t,x_0)\not=0,\;\eta(t,0)=0 \mbox{ for } t\in [t_0-\epsilon_0, t_0+\epsilon_0].
\]
Therefore we can use the zero number result of Angenent \cite{Ang}  (as stated in Lemma 2.1 of \cite{DLZ})  to conclude that,
for  $t\in (t_0-\epsilon_0,t_0+\epsilon_0)$, $\mathcal{Z}(t)$, the number of zeros of $\eta(t,\cdot)$ in $[x_0, 0]$, is finite
and nonincreasing in $t$, and if $\eta(t,\cdot)$ has a degenerate zero in $[x_0,0]$, then 
\[
\mathcal{Z}(t_2)\leq \mathcal{Z}(t_1)-1\mbox{ for  } t_0-\epsilon_0<t_1<t<t_2<t_0+\epsilon_0.
\]
It follows that there can be at most finitely many values of $t\in (t_0-\epsilon_0, t_0+\epsilon_0)$ such that $\eta(t,\cdot)$ has a degenerate zero in $[x_0,0]$. On the other hand, from
\[
\hat w_x(t,0)=-\frac{c}{\mu}=V_*'(L_*) \mbox{ for all } t\in\bR^1,
\]
we see that $x=0$ is a degenerate zero of $\eta(t,\cdot)$ for every $t\in (t_0-\epsilon_0, t+\epsilon_0)$. This contradiction proves  \eqref{hat-w}.

Since $\hat w(t,x)=V_*(x+L_*)$ is uniquely determined, we conclude that 
\bes\lbl{w-V_*}
\lim_{t\to\infty} w(t,\cdot)=V_*(\cdot+L_*) \mbox{ in } C_{loc}^{1+\alpha}(-\infty, 0].
\ees
In view of $\lim_{t\to\infty}[h(t)-ct]=L_*$,
it follows that, for every $M>0$,
\bes\lbl{h-M-h}
\lim_{t\to\infty}\left[\max_{ct-M\leq x\leq h(t)}\big|u(t,x)-V_*(x-h(t)+L_*)\big|\right]=0.
\ees
Since $V_*(-\infty)=0$ and by Lemma \ref{0-ct-M}, 
\[
\limsup_{t\to\infty}\left[\max_{0\leq x\leq ct-M}u(t,x)\right]=\epsilon(M)\to 0 \mbox{ as } M\to\infty,
\]
we have
\[
\limsup_{t\to\infty}\left[\max_{0\leq x\leq ct-M}\big|u(t,x)-V_*(x-h(t)+L_*)\big|\right]=:\tilde\epsilon(M)\to 0 \mbox{ as } M\to\infty.
\]
Combining this with \eqref{h-M-h}, we obtain, for every $M>0$,
\[
\limsup_{t\to\infty}\left[\max_{0\leq x\leq h(t)}\big|u(t,x)-V_*(x-h(t)+L_*)\big|\right]= \tilde \epsilon(M).
\]
 Letting $M\to\infty$, we obtain \eqref{0-h}.
The proof is complete.
\end{proof}

\section{Parameterized initial function and trichotomy}
\setcounter{equation}{0}

Throughout this section, we suppose 
\[
0<c<c_0.
\]
 Denote
$$
\mathcal{X}(h_0):=\big\{\phi\in C^2([0,h_0]):\phi \mbox{ satisfies \eqref{u_0}}\big\}.
$$
Fix $\phi\in\mathcal{X}(h_0)$, and for $\sigma>0$, let  $(u_\sigma,h_\sigma)$ denote the unique positive solution of (\ref{MP}) with  initial value $u_0=\sigma\phi$. We will examine the long-time dynamical behavior of $(u_\sigma, h_\sigma)$ as $\sigma$ varies, and see how the trichotomy described in Theorem \ref{tri} is realized.

We will say ``$(u_\sigma, h_\sigma)$ is a vanishing solution'', or simply ``$(u_\sigma, h_\sigma)$ is vanishing'', if in Theorem \ref{tri} case (i) vanishing happens  for $(u_\sigma, h_\sigma)$. We similarly define the terms
``$(u_\sigma, h_\sigma)$ is a borderline spreading solution'', ``$(u_\sigma, h_\sigma)$ is a spreading solution'',
``$(u_\sigma, h_\sigma)$ is  borderline spreading'', and ``$(u_\sigma, h_\sigma)$ is  spreading''.

The following result is a direct consequence of the comparison principle.
\begin{lem}\lbl{v-s}
$(i)$\, If $(u_{\sigma^1},h_{\sigma^1})$ is vanishing, then for $0<\sigma\leq \sigma^1$, $(u_{\sigma},h_{\sigma})$ is also vanishing.

$(ii)$\, If $(u_{\sigma^1},h_{\sigma^1})$ is spreading, then for $\sigma\geq \sigma^1$, $(u_{\sigma},h_{\sigma})$ is also spreading.
\end{lem}
Denote
\[
S_1:=\{\sigma> 0: (u_\sigma,h_\sigma)\ \mbox{is vanishing}\},\; \; S_2:=\{\sigma>0: (u_\sigma,h_\sigma)\ \mbox{is speading}\},
\] and
\[
\sigma_*=\left\{\begin{array}{ll}\sup S_1, & \mbox{ if } S_1\not=\emptyset,\smallskip\\
0,& \mbox{ if } S_1=\emptyset,
\end{array}\right.\;\;\;
\; \sigma^*=\left\{\begin{array}{ll}\inf S_2, & \mbox{ if } S_2\not=\emptyset,\smallskip\\
+\infty,& \mbox{ if } S_2=\emptyset.
\end{array}\right.
\]

\begin{lem}\lbl{sigma-stars}
We always have $\sigma^*\geq \sigma_*>0$. If $h_0\geq L(0)$, then $\sigma^*<+\infty$.
\end{lem}
\begin{proof}
For any fixed $T>0$, it is easy to show that 
\[
\lim_{\sigma\to 0^+}h_\sigma(T)=h_0,\; \lim_{\sigma\to
 0^+}\left[\max_{0\leq x\leq h_\sigma(T)}u_\sigma(T,x)\right]=0.
\]
Set $m:=\frac{d\pi}{9\mu}$ and fix $T>0$ such that $cT>l_0+3h_0$. We can then choose a sufficiently small $\sigma>0$ such that
\[
3 h_\sigma(T)<cT-l_0 \mbox{ and }
u_\sigma(T,x)\leq \frac{\sqrt{2}}{2}m\ \mbox{for}\ x\in[0,h_\sigma(T)].
\]
Set  
\[
\alpha:=d\frac{\pi^2}{36 h_\sigma(T)^2},
\]
\[
\tilde u_\sigma(t,x):=me^{-\alpha t}\cos\left(\frac{\pi}{2}\frac{x}{\tilde h_\sigma(t)}\right),\ \
\tilde h_\sigma(t):=h_\sigma(T)(3-e^{-\alpha t}).
\]
For $0<x\leq \tilde h_\sigma(t)$ and $t\geq T$, we have
\[
x-ct\leq \tilde h_\sigma(t)-ct\leq 3h_\sigma(T)-cT\leq -l_0.
\]
Therefore for such $t$ and $x$, $A(x-ct)=a_0$, and  the calculations in 
 the proof of Lemma \ref{vanishing-h} can be repeated to show, by the comparison principle,
\[
h_\sigma(T+t)\leq \tilde h_\sigma(t) \mbox{ for } t>0,
\]
which implies that $(u_\sigma, h_\sigma)$ is vanishing. This proves $\sigma_*>0$.

The fact $ \sigma_*\leq\sigma^*$ clearly follows from their definitions and Lemma \ref{v-s}.

Finally,
if  $h_0\geq L(0)$, then we can find $\sigma>0$ large enough such that
\[
\sigma\phi(x)\geq V_0(x) \mbox{ for } x\in [0, L(0)].
\]
Therefore by Remark \ref{suff-cond}, $(u_\sigma, h_\sigma)$ is spreading.
It follows that $\sigma^*<+\infty$.
\end{proof}

\begin{theo}\lbl{vanishing-spreading}\,
There exists $\sigma_0\in(0,\infty]$ such that
\begin{itemize}
\item[(i)]\, $(u_\sigma, h_\sigma)$ is vanishing when $\sigma<\sigma_0;$
\item[(ii)]\, $(u_\sigma, h_\sigma)$ is spreading when $\sigma>\sigma_0;$
\item[(iii)]\, $(u_\sigma, h_\sigma)$ is borderline spreading when $\sigma=\sigma_0$.
\end{itemize}
\end{theo}
\begin{proof}\, For clarity, we divide the proof into three steps.

\noindent
{\bf Step 1.} \; We show that $\sigma_*\not\in S_1$.

Arguing indirectly, we suppose that $\sigma_*\in S_1$. Then  we have
\[\mbox{
$\lim_{t\rightarrow\infty}\left[\max_{0\leq x\leq h_{\sigma_*}(t)}u_{\sigma_*}(t,x)\right]=0$ and $(h_{\sigma_*})_\infty<\infty$.}
\]
Denote $m=\frac{d\pi}{9\mu}$. Then
there exists $T>0$ such that
$$
3h_{\sigma_*}(T)<cT-(l_0+1)\ \mbox{and}\ \max_{0\leq x\leq h_{\sigma_*}(T)}u_{\sigma_*}(T,x)\leq \sqrt{2}m/4.
$$
By continuity of solutions with respect on $\sigma$, we may choose a small $\epsilon>0$ such that
the corresponding solution $(u_{\sigma_*+\epsilon},h_{\sigma_*+\epsilon})$ satisfies
\[
3h_{\sigma_*+\epsilon}(T)<cT-l_0\ \mbox{and}\ \max_{0\leq x\leq h_{\sigma_*+\epsilon}(T)}u_{\sigma_*+\epsilon}(T,x)\leq\sqrt{2}m/2.
\]
Let 
\[
\alpha:=\frac{d\pi^2}{36h_{\sigma_*+\epsilon}(T)^2}
\]
and define
\[
\xi(t):=h_{\sigma_*+\epsilon}(T)(3-e^{-\alpha t}), \ \omega(t,x):=me^{-\alpha t}\cos\left(\frac{\pi}{2}\frac{x}{\xi(t)}\right).
\]
Then by the arguments in the proof of Lemma \ref{vanishing-h}, we have
\[
h_{\sigma_*+\epsilon}(T+t)\leq \xi(t) \ \mbox{for}\ t>0.
\]
This implies $\lim_{t\to\infty}h_{\sigma_*+\epsilon}(t)<\infty$,
and hence  $(u_{\sigma_*+\epsilon},h_{\sigma_*+\epsilon})$ is vanishing, which is a
contradiction to the definition of $\sigma_*$. Therefore, $\sigma_*\not\in S_1$, and Step 1 is completed.

Let us note that by Lemma \ref{v-s},  vanishing happens when $0<\sigma<\sigma_*$. If $\sigma_*=+\infty$, then there is nothing left to prove. We suppose next $\sigma_*<+\infty$.
\smallskip

\noindent
{\bf Step 2.} \ We show that $\sigma^*\not\in S_2$.

Suppose that $\sigma^*\in S_2$. Then, by the definition of spreading, we have $h_{\sigma^*}(t)-ct\rightarrow\infty$ as $t\rightarrow\infty$.
Fix a constant $\tilde l>\max\{h_0, L_*\}$; then there exists $t_0>0$ such that
\[
h_{\sigma^*}(t_0)-ct_0>L(0)+\tilde l+1.
\]
By continuity, there exists a sufficiently small $\epsilon>0$ such that
\bes
\lbl{t-0}
h_{\sigma^*-\epsilon}(t_0)-ct_0>L(0)+\tilde l.
\ees
Define
\[
\mbox{$
w(t,x):=V_{0}(x-ct-\tilde l)$ for $t>0$ and $ct+\tilde l\leq x\leq ct+\tilde l+L(0)$.}
\]
In view of \eqref{t-0} and the fact that 
\[
h_{\sigma^*-\epsilon}(0)=h_0<\tilde l,
\]
we can find $t_1$ and $t_2$ such that $0<t_2<t_1<t_0$,
\[
h_{\sigma^*-\epsilon}(t_2)-ct_2=\tilde l,\; h_{\sigma^*-\epsilon}(t_1)-ct_1=\tilde l+L(0),
\]
and
\[
\tilde l<h_{\sigma^*-\epsilon}(t)-ct< \tilde l+L(0) \mbox{ for } t\in (t_2, t_1).
\]
We are now in a position to repeat the zero number argument
in the proof of Lemma \ref{prop-infty}, to deduce that
$$
V_0(x-ct_1-\tilde l)=w(t_1,x)<u_{\sigma^*-\epsilon}(t_1,x) \ \mbox{for}\ ct_1+\tilde l\leq x<h_{\sigma^*-\epsilon}(t_1),
$$
and
$$
h_{\sigma^*-\epsilon}(t_1)=ct_1+\tilde l+L(0).
$$
By the comparison principle for free boundary problems, we then deduce
\[
h_{\sigma^*-\epsilon}(t)\geq ct+\tilde l+L(0) \mbox{ for } t>t_1,
\]
and 
\[
u_{\sigma^*-\epsilon}(t,x)\geq V_0(x-ct-\tilde l) \mbox{ for } t>t_1,\; x\in [ct+\tilde l, ct+\tilde l+L(0)].
\]
It follows that 
\[
\limsup_{t\to\infty}[ h_{\sigma^*-\epsilon}(t)-ct]\geq \tilde l+L(0)>L_*.
\]
Hence we can use Theorem \ref{1-2-3} to conclude that $(u_{\sigma^*-\epsilon}, h_{\sigma^*-\epsilon})$ is spreading,
 which is a contradiction to the
definition of $\sigma^*$. Therefore, $\sigma^*\not\in S_2$, and Step 2 is done.

By Lemma \ref{v-s}, $(u_\sigma, h_\sigma)$ is
spreading  when $\sigma>\sigma^*$. Moreover, for any $\sigma\in [\sigma_*, \sigma^*]\cap \bR^1$, $(u_\sigma, h_\sigma)$ is
not vanishing, nor spreading, so by Theorem \ref{tri}, $(u_\sigma, h_\sigma)$ must be borderline spreading.

\smallskip

\noindent
{\bf Step 3.}
\ We prove that $\sigma_*=\sigma^*$.

 Suppose that $\sigma_*<\sigma^*$.  For convenience, we denote
\[
(u_*,h_*)=(u_{\sigma_*},h_{\sigma_*}) \mbox{ and } (u^*,h^*)=(u_{\sigma^*},h_{\sigma^*}). 
\]
(If $S_2=\emptyset$ and hence $\sigma^*=+\infty$, then  we take $\sigma^*$  an arbitrary number in $(\sigma_*, +\infty)$
in the definition of $(u^*, h^*)$ above.)
Then  both $(u_{*},h_{*})$ and $(u^*,h^*)$ are borderline spreading, and so
\[
\lim_{t\rightarrow\infty}[h_*(t)-ct]=\lim_{t\rightarrow\infty}[h^*(t)-ct]=L_{*}.
\]
By the comparison principle,
\bes\lbl{u_*-u^*}
u_*(t,x)<u^*(t,x)\ \ \mbox{for}\ t>0,\ 0\leq x\leq h_*(t)
\ees
and
\bes\lbl{h_*-h^*}
h_*(t)<h^*(t)\ \ \mbox{for}\ t>0.
\ees
From the proof of Lemma \ref{lem3}, we see that $V_{*}'(x)>0$ for $x<-l_0$. Hence we can use \eqref{w-V_*} to conclude that
for all large $t$, say $t\geq t_0>0$, 
\bes\lbl{u^*_x}
u^*_x(t,x)>0\ \ \mbox{for}\ t\geq t_0, \ -2l_0\leq x-ct\leq -\frac{3}{2}l_0.
\ees
We may also assume that $ct_0>2l_0$.

By (\ref{u_*-u^*}) and (\ref{h_*-h^*}), there is small $\tau_0>0$ such that
\[
u_*(t_0,x-\tau_0)<u^*(t_0,x)\ \ \mbox{for}\ ct_0-2l_0+\tau_0\leq x\leq h_*(t_0)+\tau_0
\]
and
\[
h_*(t_0)+\tau_0<h^*(t_0).
\]
Define
\[
\underline u(t,x):=u_*(t+t_0,x-\tau_0),
\]
\[
\xi_1(t):=c(t+t_0)-2l_0+\tau_0,\ \xi_2(t):=h_*(t+t_0)+\tau_0.
\]
Then,
\bess
\underline u_t&=&d\underline u_{xx}+A(x-\tau_0-c(t+t_0))\underline u-b\underline u^2\\
&\leq&d\underline u_{xx}+A(x-c(t+t_0))\underline u-b\underline u^2;
\eess
\[
\underline u(0,x)=u_*(t_0,x-\tau_0)<u^*(t_0,x),\ \ x\in[\xi_1(0),\xi_2(0)];
\]
\[
-\mu\underline u_x(t,\xi_2(t))=-\mu\frac{\partial u_*}{\partial x}(t+t_0,h_*(t+t_0))=h_*'(t+t_0)=\xi_2'(t);
\]
\[
\underline u(t,\xi_2(t))=0,\ \ t>0.
\]
By (\ref{u_*-u^*}) and (\ref{u^*_x}), for $t>0$,
\[
\underline u(t,\xi_1(t))=u_*(t+t_0,c(t_0+t)-2l_0)<u^*(t+t_0,c(t+t_0)-2l_0)<u^*(t+t_0,\xi_1(t)).
\]
By the comparison principle,
\[
\underline u(t,x)\leq u^*(t+t_0,x)\ \ \mbox{for}\ t>0,\ x\in[\xi_1(t),\xi_2(t)]
\]
and
\bes\lbl{h_*<<h^*}
\xi_2(t)\leq h^*(t+t_0)\ \ \mbox{for}\ t>0.
\ees
From (\ref{h_*<<h^*}) and $\lim_{t\rightarrow\infty}[h_*(t)-ct]=L_{*}$, we  obtain
\[
\liminf_{t\rightarrow\infty}[h^*(t)-ct]\geq \lim_{t\to\infty}[\xi_2(t-t_0)-ct]= L_{*}+\tau_0>L_*,
\]
which contradicts $\lim_{t\rightarrow\infty}[h^*(t)-ct]=L_{*}$.
This completes Step 3, and hence  the proof of the theorem.
\end{proof}

\section{Proof of Theorem \ref{c>c_0}}
\setcounter{equation}{0}

 We will consider the cases $c>c_0$ and $c=c_0$ separately. We start with the easy case $c>c_0$.

\begin{lem}
If $c>c_0$, then the unique solution $(u,h)$ of \eqref{MP} is always vanishing.
\end{lem}
\begin{proof}
Since $A(x-ct)\leq a$, by the comparison principle and \cite{DMZ}, there are $t_0>0$ and $\delta>0$ such that
$h(t)<(c-\delta)t$ for $t>t_0$. By Lemma \ref{0-ct-M}, we have
\[
\lim_{t\rightarrow \infty}\left[\max_{0\leq x\leq h(t)}u(t,x)\right]=0.
\]
We may now apply Lemma \ref{vanishing-h} to obtain $h_\infty<+\infty$. Hence $(u,h)$ is vanishing.
\end{proof}

We next treat the case $c=c_0$.

\begin{lem}\lbl{c=k}\,
When $c=c_0$,  the unique solution $(u,h)$ of \eqref{MP} is always vanishing.
\end{lem}
\begin{proof}\,
 We understand that $u(t,x)=0$ for $x>h(t)$.
By the comparison principle and \cite{DMZ}, there exists $L>0$ such that
\[
h(t)-ct<L\ \mbox{for}\ t\geq 0.
\]
Denote $v(t,x):=u(t,x+ct)$; then
\bess\left\{\begin{array}{ll}\medskip
\displaystyle v_t=dv_{xx}+cv_x+A(x)v-bv^2, \ \ & t>0,\ -ct<x<h(t)-ct,\\
\displaystyle v(t,h(t)-ct)=0,\ \ & t>0.
\end{array}\right.
\eess
By Lemma \ref{0-ct-M}, there exist $t_0>0$ and $M_0>0$ such that
\bes\lbl{-ct-M}
v(t,x)<\frac{a}{2b} \mbox{ for } t\geq t_0,\; x\in [-ct, -M_0].
\ees

 Let $M=\max\{\|u_0\|_\infty,a/b\}$. Then consider the following problem
\bes\lbl{v-1-equation}\left\{\begin{array}{ll}\medskip
\displaystyle w_t=dw_{xx}+cw_x+A(x)w-bw^2, \ \ & t>t_0,\ -M_0<x<L,\\
\medskip\displaystyle w(t,-M_0)=\frac{a}{2b},\ \ w(t,L)=0,\ \ & t>t_0,\\
\displaystyle w(t_0,x)=M,\ \ &     -M_0\leq x\leq L.
\end{array}\right.
\ees
Since $w=M$ is a super-solution of the corresponding elliptic problem of \eqref{v-1-equation},
by a well-known result on parabolic equations (see \cite{S}), 
the unique solution of (\ref{v-1-equation}), which we denote by $v^1(t,x)$, is decreasing in $t$ and 
\[
\lim_{t\to\infty} v^1(t,\cdot)=V^1 \mbox{ in } C^2([-M_0, L]),
\]
with $V^1$  the unique positive solution of 
\bess\left\{\begin{array}{ll}\medskip
\displaystyle dV''+cV'+A(x)V-bV^2=0, \ \  -M_0<x<L,\\
\displaystyle V(-M_0)=\frac{a}{2b},\ \ V(L)=0.
\end{array}\right.
\eess
Since $h(t)-ct<L$ for all $t$, and $M\geq u_0$, and \eqref{-ct-M} holds, by the comparison principle, we have
\[
v(t,x)\leq v^1(t,x)\ \ \mbox{for}\ \ t\geq t_0,\ -M_0<x<L.
\]
By a simple comparison argument and the strong maximum principle, we see that there exists $\epsilon_1>0$ such that
\[
V^1(x)\leq \frac{a}{b}-2\epsilon_1\ \ \mbox{for}\ x\in[-M_0,L].
\]
So, there is $t_1>t_0$ such that
\bess
v(t,x)<\frac{a}{b}-\epsilon_1\ \ \mbox{for}\ t\geq t_1, \ -M_0\leq x\leq L.
\eess
Combining this with \eqref{-ct-M}, we obtain, for some $\epsilon_2>0$,
\bes\lbl{right-estimate}
u(t,x)<\frac{a}{b}-\epsilon_2\ \ \mbox{for}\ t\geq t_1, \ 0 \leq x\leq ct+L.
\ees

Since $c=c_0$, by \cite{DL, BDK}, the problem
\bes\lbl{U-c-L}\left\{\begin{array}{ll}\medskip
\displaystyle dU''+cU'+aU-bU^2=0, \ \ -\infty<x<0,\\
\medskip\displaystyle U(0)=0,\; U(-\infty)=\frac{a}{b},\\
\displaystyle-\mu U'(0)=c
\end{array}\right.
\ees
has a unique positive solution $U_{c}$. Define 
\[
U_{c, L_1}(x)=U_c(x-L_1).
\]
If we choose $L_1>L$ large enough, then from $U_c(-\infty)=a/b$ we obtain
\[
U_{c,L_1}(x)>\frac{a}{b}-\frac{\epsilon_2}{4}\ \ \mbox{for}\ x\in(-\infty,L].
\]
By continuity, there is sufficiently small $\delta>0$ such that 
\[
U_{c-\delta,L_1}(x)>\frac{a}{b}-\frac{\epsilon_2}{2}\ \ \mbox{for}\  x\in[0, L],
\] 
where
$U_{c-\delta,L_1}(x):=U_{c-\delta}(x-L_1)$, and $U_{c-\delta}$ is the unique solution of the initial value problem
\[
dU''+(c-\delta)U'+aU-bU^2=0,\; U(0)=0,\; U'(0)=-\frac{c-\delta}{\mu}.
\]

Define
\[
\tilde u(t,x):=U_{c-\delta,L_1}(x-ct_1-(c-\delta)t),
\]
\[
\xi_1(t):=ct_1+(c-\delta)t,\ \xi_2(t):=ct_1+(c-\delta)t+L_1.
\]
Then, for $ t>0$ and $ \xi_1(t)<x<\xi_2(t)$,
\bess
\tilde u_t&=&d\tilde u_{xx}+a\tilde u-b\tilde u^2\\
&\geq& d\tilde u_{xx}+A(x-c(t+t_1))\tilde u-b\tilde u^2.
\eess
By  (\ref{right-estimate}),
\[
\tilde u(t,\xi_1(t))=U_{c-\delta,L_1}(0)>\frac{a}{b}-\frac{\epsilon_2}{2}>u(t,\xi_1(t)) \mbox{ for } \ t>0.
\]
Obviously,
\[
\tilde u(t,\xi_2(t))=U_{c-\delta,L_1}(L_1)=0;
\]
\[
-\mu\tilde u_x(t,\xi_2(t))=c-\delta=\xi_2'(t) \mbox{ for }\ t>0.
\]
If $[\xi_1(0),h(t_1)]$ is not empty, then by (\ref{right-estimate})
\[
\tilde u(0,x)\geq \frac{a}{b}-\epsilon_2>u(t_1,x)\ \mbox{for}\ x\in[\xi_1(0),h(t_1)].
\]
Hence we can use  the comparison principle to conclude that
\[
\tilde u(t,x)\geq u(t+t_1,x)\ \mbox{for}\ t>0,\ \xi_1(t)<x<h(t_1+t)
\]
and
\[
\xi_2(t)\geq h(t+t_1)\ \mbox{for}\ t>0.
\]
We may now use Lemma \ref{0-ct-M} to conclude that
\[
\lim_{t\rightarrow\infty}\left[\max_{x\in[0,h(t)]}u(t,x)\right]=0.
\]
Therefore, $h_\infty<+\infty$ (see Lemma \ref{vanishing-h}), and vanishing happens.
\end{proof}

\end{document}